\documentclass{article}
\usepackage{latexsym}
\newcommand{\kewemmdef}{\newcommand{\kew}{q}\newcommand{\emm}{m}\renewcommand{\phi}{\varphi}\newcommand{\CITE}{\cite}\newcommand{\text}{\mbox}}
\newtheorem{theorem}{Theorem}

\newtheorem{definition}[theorem]{Definition}

\newtheorem{proposition}[theorem]{Proposition}
\newtheorem{remark}[theorem]{Remark}

\newenvironment{proof}
{\textit{Proof.} }{\hfill$\Box$\par\addvspace{\medskipamount}}
\newenvironment{completionofproof}
{\textit{Completion of proof.} }{\hfill$\Box$\par\addvspace{\medskipamount}}
\kewemmdef

\begin{document}
\title{Multiresolution wavelet analysis of integer scale Bessel functions }
\author{S. Albeverio\\
Institut f\"ur Angewandte Mathematik\\ Universit\"at Bonn\\ Wegelerstr. 6, D-53115, Bonn, Germany\\
SFB 256, Bonn, BiBos (Bielefeld-Bonn)\\
IZKS Bonn, CERFIM\\ Locarno and ACC. ARCH (USI)\\ Switzerland
\and
P.E.T. Jorgensen\\
Mathematics Department\\University of Iowa\\Iowa City, IA 52242, USA%
\and
A.M. Paolucci\\
Max-Planck-Institut f\"ur Mathematik\\Vivatsgasse 7, 53111 Bonn\\Germany
}
\date{\today}
\maketitle
AMS Subject Classification: (2000) 42C40, 37G99, 11B37, 46E22, 32G07.
Key words: mathematical methods in physics, wavelet, deformations, spectral theory,
special functions, recurrence relations, Hilbert space.\\
AIP classification: mathematical methods in physics 02.30.-f, 02.30.Mv,
02.30.Tb, 02.30.Uu, 02.50.Cw

\begin{abstract}
We identify multiresolution subspaces giving rise via Hankel transforms to
Bessel functions. They emerge as orthogonal systems derived from geometric
Hilbert-space considerations, the same way the wavelet functions from a
multiresolution scaling wavelet construction arise from a scale of Hilbert
spaces. We study the theory of representations of the $C^{\ast }$-algebra $%
O_{\nu +1}$ arising from this multiresolution analysis. A connection with Markov
chains and representations of  $O_{\nu +1}$ is found. Projection valued measures arising from
the multiresolution analysis give rise to a Markov trace for quantum groups 
$SO_q$.
\end{abstract}

\section{\label{INT}Introduction}

The starting point for the multiresolution analysis from wavelet theory is a
system $U$, $\left\{ T_{j}\right\} _{j\in {\bf Z}}$, of unitary operators
with the property that the underlying Hilbert space ${\cal H}$ with norm $\left\| .\right\|$,
contains a
vector $\phi \in {\cal H}$, $\left\| \phi \right\| =1$, satisfying
\begin{equation}
U\phi =\sum_{j}a_{j}T_{j}\phi \label{eqpound.1}
\end{equation}
for some sequence $\left\{ a_{j}\right\} $ of complex scalars,  such that, in
particular (\ref{eqpound.1}) converges in ${\cal H}$.
In addition, the operator system $\left\{ U,T_{j}\right\} $ must
satisfy a non-trivial commutation
relation. In the case of wavelets,
it is
\begin{equation}
UT_{j}U^{-1}=T_{N},\qquad j\in{\bf Z}, \label{eqpound}
\end{equation}
where $N$ is the \emph{scaling number}, or
equivalently the number of \emph{subbands}
in the corresponding multiresolution.
When this structure is present, there is a way to recover the
spectral theory of the problem at hand from representations of an associated
$C^{\ast}$-algebra. In the case of
orthogonal wavelets, we may take this
$C^{\ast}$-algebra to be the Cuntz algebra.
In that case, the operators $T_{j}$ may be represented
on $L^{2}\left( {\bf R}\right) $ as translations,
\[
\left( T_{j}\xi\right) \left( x\right) =\xi\left( x-j\right) ,\qquad
\xi\in L^{2}\left( {\bf R}\right) ,
\]
and $U$ may be taken as the scaling $\left( U\xi\right) \left( x\right) =
N^{-1/2}\xi\left( x/N\right)$, $N \in \textbf{N} $.
This system clearly satisfies (\ref{eqpound}).
(For a variety of other examples of these relations, the
reader is referred to Ref.\ \CITE{Jor01}.
The setup there applies to dynamical
systems of $N$-to-$1$ Borel measurable self-maps:
for example, those of complex
dynamics and Julia sets.)
In the wavelet case, a
multiresolution is built from
a solution $\phi\in L^{2}\left( {\bf R}\right) $ to the
scaling identity (\ref{eqpound.1}). The numbers
$\left\{ a_{j}\right\} _{j\in {\bf Z}}$ from (\ref{eqpound.1}) must then
satisfy the "orthogonality relations"'
\begin{equation}
\sum_{k\in\mathbf{Z}}a_{k}=1,\qquad\sum_{k\in\mathbf{Z}}\bar{a}_{k}a_{k+2m}=\delta_{0,m}, m\in\textbf Z
\label{eqpound.2}
\end{equation}
In this case, the analysis is based
on the Fourier transform: define $m_{0}$ as a map from $S^1$  to $\textbf{C}$ by
\begin{equation}
m_{0}\left( e^{it}\right) =\sum_{k}a_{k}e^{ikt}, t\in\textbf{R}
\label{eqpound.3}
\end{equation} 
(of course we assume here and below convergence of the series and products involved).
Then (in the wavelet case, following Ref.\ \CITE{dau})
a solution to (\ref{eqpound.1}) will have the product form
\begin{equation}
\hat{\phi}\left( t\right) =\prod_{j=1}^{\infty}m_{0}\left( t/N^{j}\right) ,
\label{eqpound.4}
\end{equation}
up to a constant multiple. The Cuntz algebra $O_{N}$
enters the picture as follows: Formula (\ref{eqpound.4})
is not practical for computations, and the analysis of orthogonality
relations is done better by reference to the Cuntz relations,
see (\ref{eqCunpound.1})--(\ref{eqCunpound.2}) below.
Setting, for $\xi\in\textbf{C}$, $j\in\textbf{Z}$
\begin{equation}
W\left( \left\{ \xi_{j}\right\} \right) :=\sum_{j\in\textbf{Z}}\xi_{j}
\phi\left( x-j\right) ,
\label{eqpound.5}
\end{equation}
and using (\ref{eqpound.2}), we get an
isometry $W$ of $\ell^{2}$ into a subspace
of $L^{2}\left( {\bf R}\right) $, the
resolution subspace. Setting
\begin{equation}
\left( S_{0}f\right) \left( z\right) :=\sqrt{N}m_{0}\left( z\right)
f\left( z^{N}\right) ,\qquad f\in L^{2}\left( {\bf T}\right) , \rm{Borel\ measurable}
\label{eqpound.6}
\end{equation}
and using $L^{2}\left( {\bf T}\right) \cong\ell^{2}$
by the Fourier series, we establish
the following crucial intertwining identity:
\begin{equation}
WS_{0}=UW,
\label{eqpound.7}
\end{equation}
so that $U$ is a unitary extension
of the isometry $S_{0}$. We
showed in Refs.\ \CITE{BEJ} and \CITE{Br-Jo2} that
functions $m_{1},\dots,m_{N-1}\in L^{\infty}\left( {\bf T}\right) $
may then be chosen such that the corresponding matrix
\begin{equation}
\left( m_{j}\left( e^{i\left( t+k2\pi /N\right) }\right) \right)
_{j,k=0}^{N-1}
\label{eqpound.8}
\end{equation}
is in ${\rm U}_{N}\left( {\bf C}\right) $ for Lebesgue a.a.\ $t$.
Then it follows that the operators
\begin{equation}
S_{j}f\left( z\right) :=\sqrt{N}m_{j}\left( z\right)
f\left( z^{N}\right) ,\qquad f\in L^{2}\left( {\bf T}\right) ,
\label{eqpound.9}
\end{equation}
will yield a representation of the Cuntz relations;
see (\ref{eqCunpound.1})--(\ref{eqCunpound.2}) below. Conversely, if
(\ref{eqpound.9}) is given to satisfy the Cuntz relations, then the
matrix in (\ref{eqpound.8}) takes values in
${\rm U}_{N}\left( {\bf C}\right) $.

The present paper aims at an analogous construction, but
based instead on the Bessel functions, i.e., we use the
Bessel functions in (\ref{eqpound.3}) in place of the usual
Fourier basis $\left\{ e^{ikt}\right\} _{k\in{\bf Z}}$;
see (\ref{m0}) below.
If $\nu\in\textbf N$ is the parameter of the Bessel function $J_{\nu}$,
then we show that $N=\nu +1$ is an admissible scaling
for a multiresolution construction.

The motivation for doing a
multiresolution construction based on
a wider variety of special
functions, other than the Fourier
basis, derives in part from the
rather restrictive axiom system
dictated by the traditional setting \CITE{{BJMP05}, {Hut81}, {Jor01}, {Jor06},{JoKr03},{LaWa96}}.
It is namely known\cite{dau} that
many applications require
a more general mathematical setup. Moreover,
our present approach also throws some new
light on special-function theory, and
may be of independent interest for
that reason.

We will apply multiresolutions to the Hankel transform and the
Bessel functions of integer parameter $\nu $. Our analysis is
especially well suited for the introduction of a quantum variable $q$,( $0<q<1$), in
such a way that variations in $q$ lead to a better understanding of an
associated family of deformations.
Our use of the Cuntz algebra is motivated by
Refs.\ \CITE{BEJ} and \CITE{Br-Jo}. The Cuntz algebras \cite{Cu} have
been used independently in operator algebra theory
and in the study of multiresolution wavelets,
and our present paper aims to
both make this connection explicit, and as
well make use of it in the analysis of
special functions. The $q$-deformations
of the special functions\cite{Bie,Ch,Di,Ex,Is,McF}
may be of independent interest.
This deformation is related to,
but different from, those which
have appeared in Refs.\ \CITE{McF,JSW,Jo-We,K-VA,RS}.
In the last sections of the paper we construct a Markov chain which turns out
to be related to the representations of $O_{\nu +1}$ discussed in the previous sections
via  projection valued measures. Random walks on quantum group $SO_q(N)$ are then
constructed via representations of the braid groups.

\section{\label{CUN}The Cuntz algebra and iterated function systems}

We
shall\textit{}
consider representations $\pi $ of the Cuntz algebra $O_{\nu +1}$ coming
from multiresolution analysis based on Hankel transforms. In Section
\ref{HAN}
we
give some preliminaries on Hankel transforms on $L^{2}\left( {\bf R}\right) $%
. We then construct wavelets arising from multiresolutions with scaling $\nu
+1$ using Hankel transforms on $L^{2}\left( {\bf C}\right) $, relative to an
appropriate measure on the field of complex numbers ${\bf C}$. The map from wavelets into representations
is described. We establish connections between certain representations of $%
O_{\nu +1}$ and Hankel wavelets arising from that multiresolution analysis.

Recall that $O_{\nu +1}$ is the $C^{\ast }$-algebra generated by $\nu +1$, $%
\nu \in {\bf N}$, isometries $S_{0},\dots ,S_{\nu }$ satisfying
\begin{equation}
S_{i}^{\ast }S_{j}=\delta _{ij}{\bf 1} \label{eqCunpound.1}
\end{equation}
and
\begin{equation}
\sum_{i=0}^{\nu }S_{i}S_{i}^{\ast }={\bf 1}. \label{eqCunpound.2}
\end{equation}
The representations we will consider are realized on the Hilbert spaces
$H=L^{2}\left( \Omega ,d\mu \right) $ where $\Omega $ is a measure space (to
be specified below) and $\mu $ is a probability measure on $\Omega $.

We define the representations in terms of certain maps
\begin{equation}
\sigma _{i}\colon \Omega \longrightarrow \Omega
\quad \mbox{such that}\quad
\Omega = \bigcup_{i=0}^{\nu}\sigma_{i}\left( \Omega \right)
\quad \mbox{and}\quad
\mu \left( \sigma _{i}\left( \Omega \right) \cap \sigma _{j}\left( \Omega
\right) \right) =0
\label{eqCunpound.3}
\end{equation}
for all $i\neq j$. We will apply this in Section \ref{MUL} to the Riemann surface of
$\sqrt[N]{z}$.

In Section
\ref{KEW}
we develop a $q$-parametric multiresolution wavelet analysis in $%
L^{2}\left( {\bf C},\mu _{q}\right) $ where $\mu _{q}$ is a $q$-measure,
as in Refs.\ \CITE{RS, Ga-Ra} by using $q$-Hankel transforms.

A class of
$q$-parametric
representations of the $C^{\ast }$-algebra $O_{\nu +1}$
is found. We further identify a class of representations of the Cuntz algebra
which has the structure of compact quantum groups of type B.\CITE{Pa-La}

\section{\label{HAN}Hankel transforms and a multiresolution analysis}

In this section we construct a multiresolution using Hankel transforms. We
start by giving some basic definitions on Hankel tranforms.

Let us recall that the Hankel transform of order $\alpha \in \textbf R$ of a function $f$%
, denoted by $\tilde{f}$, is defined, for $t\in\ ( 0,\infty )$ and $x\in ( 0,\infty )$, by
\begin{equation}
\tilde{f}\left( t\right) =\int_{0}^{\infty }J_{\alpha }\left( xt\right)
f\left( x\right) x\,dx ,   \label{e1.1}
\end{equation}%
where
\[
J_{\alpha }\left( x\right) =\left( \frac{x}{2}\right) ^{\alpha
}\sum_{k=0}^{\infty }\frac{\left( -1\right) ^{k}}{k!\Gamma \left( \alpha
+k+1\right) }\left( \frac{x}{2}\right) ^{2k}
\]%
is the Bessel function of order $\alpha $, $\alpha\in\textbf{R} $ and
\[
\Gamma \left( z\right) =\int_{0}^{\infty }e^{-t}t^{z-1}\,dt,\qquad
\mathop{\rm Re}%
\left( z\right) >0,
\]%
is the classical gamma function.
If we multiply both sides of (\ref{e1.1}) by $J_{\alpha }\left( yt\right) t$
and integrate from $t=0$ to $+\infty $ we obtain%
\begin{equation}
\int_{0}^{\infty }J_{\alpha }\left( yt\right) \tilde{f}\left( t\right)
t\,dt=f\left( y\right) =\int_{0}^{\infty }J_{\alpha }\left( yt\right)
t\int_{0}^{\infty }J_{\alpha }\left( xt\right) f\left( x\right) x\,dx\,dt, y\in ( 0,\infty )
\label{e1.2}
\end{equation}%
The integral transform on the left-hand side of (\ref{e1.2}) is equal to $%
f\left( y\right) $ for suitable functions $f$, by the Hankel inversion
theorem.\cite{IS} The resulting double integral is called the Hankel
Fourier-Bessel integral
\begin{equation}
f\left( y\right) =\int_{0}^{\infty }J_{\alpha }\left( yt\right) \left(
\int_{0}^{\infty }J_{\alpha }\left( xt\right) f\left( x\right) x\,dx\right)
t\,dt.
\label{eA}
\end{equation}%
It can be written as the following transform pair
\begin{eqnarray}
g\left( t\right)  &=&\int_{0}^{\infty }J_{\alpha }\left( yt\right) f\left(
y\right) y\,dy,  \label{e2.1'} \\
f\left( y\right)  &=&\int_{0}^{\infty }J_{\alpha }\left( yt\right) g\left(
t\right) t\,dt.  \nonumber
\end{eqnarray}

A Plancherel type result can be easily derived for this transform: if $%
F(\rho )$ and $G(\rho )$, $\rho\in ( 0,\infty )$, are Hankel transforms of $f(x)$ and $g(x)$, $x\in ( 0,\infty )$,
respectively, then we have%
\begin{eqnarray*}
\int_{0}^{\infty }\rho F(\rho )G(\rho )\,d\rho &=&\int_{0}^{\infty }\rho
F(\rho )\int_{0}^{\infty }xg(x)J_{\nu }(\rho x)\,dx\,d\rho \\
&=&\int_{0}^{\infty }xg(x)\left( \int_{0}^{\infty }\rho F(\rho )J_{\nu
}\left( \rho x\right) \,d\rho \right) \,dx\\
&=&\int_{0}^{\infty }xf(x)g(x)\,dx.
\end{eqnarray*}

Let us give some preliminaries on the standard multiresolution wavelet
analysis of scale $\nu $, $\nu \in {\bf N}$. Following Refs.\
\CITE{dau,BJ}
we define scaling by $\nu $ on $L^{2}\left( {\bf R}\right) $ by
\[
\left( U\xi \right) \left( x\right) =\left( \nu +1\right) ^{-\frac{1}{2}}\xi
\left( \frac{x}{\nu +1}\right)
\]
and translation by $1$ on $L^{2}\left( {\bf R}\right) $ by
\[
\left( T\xi \right) \left( x\right) =\xi \left( x-1\right) , x\in\textbf R.
\]

As mentioned in the Introduction,
it is our aim here to adapt
the theory of multiresolutions from
wavelet theory\cite{dau,ma} to the
analysis of the Bessel functions via the
Hankel transform. The classical theory\cite{As-Is}
is based on recurrence
algorithms which we show adapt
very naturally to the multiresolutions.
But our analysis will still be
based on the ``classical'' identities
for the special functions
(see, e.g., Refs.\ \CITE{Ga-Ra,Hendriksen,K-S,Pou,VA,W,Is2}).

A scaling function is a Borel measurable function $\phi \in L^{2}\left( {\bf R}\right) $ such
that if $V_{0}$ is the closed linear span of all translates $T^{k}\phi $, $%
k\in {\bf Z}$, then $\phi $ has the following four properties

\begin{enumerate}
\item[i)] $\left\{ T^k\phi :k\in {\bf Z}\right\} $ is an orthonormal set in $%
L^2\left( {\bf R}\right) $;

\item[ii)] $U\phi \in V_0$;

\item[iii)] $\bigwedge _{n\in {\bf Z}}U^{n}V_{0}=\left\{ 0\right\} $;

\item[iv)] $\bigvee _{n\in {\bf Z}}U^{n}V_{0}=L^{2}\left( {\bf R}\right) $.
\end{enumerate}

The simplest example of a scaling function is the characteristic function of
the interval $\left[ 0,1\right] $, i.e., the zeroth Haar function. By i) we may
define an isometry
\[
F_{\phi }\colon V_{0}\longrightarrow L^{2}\left( {\bf R}\right) ,\qquad \xi
\longmapsto m,
\]%
as follows. The scaling by $\nu $ on $L^{2}\left( {\bf R}\right) $ is
defined by the unitary operator $U$ given by $(U\xi )(x)=
\left( \nu +1\right) ^{-\frac{1}{2}}\xi
\left( \left( \nu +1\right) ^{-1}x\right) $
for $\xi \in L^{2}\left( {\bf R}\right) $, $x\in {\bf R}$, and
the translation as the following operator $(T\xi )(x)=\xi (x-1)$.

We consider the scaling Haar function $\varphi $ given as the sum $\varphi
(x)=h\left( 1-x\right) -h\left( -x\right) $ of Heaviside functions $h$, $h(x)=1$ for $x\geq0$ and
 $h(x)=0$ for $x<0$

Let $V_{0}$ be the linear span of $\left\{ \phi _{\nu }^{\left( k\right)
}\left( x\right) \equiv x^{\nu }\varphi \left( x-k\right) \right\} _{k\in
{\bf Z}}$.
Then $V_{0}$ is a closed subspace of $L^{2}\left( {\bf R}\right) $ with
respect to the following scalar product:
$\displaystyle \langle f\mid g\rangle =\int \overline{f\left( x\right) }
g\left( x\right) x\,dx$. 
 We have $\bigcap_{n\in
{\bf Z}}U^{n}V_{0}=\{0\}$ and $\bigvee U^{n}V_{0}=L^{2}\left( {\bf R}\right) $.

Let $\xi \in L^{2}\left( {\bf R}\right) $, and assume that $\xi
(x)=\sum_{k}b_{k}\left\{ \phi _{\nu }^{\left( k\right)
}\left( x\right)\right\}$, $b_{k}\in\textbf{C}$, $x\in\textbf{R}$

By applying the Hankel transform $H_{\nu }(\,\cdot \,,t)$ to both sides of
the above equality and using the definition of $\phi$ we get for $t\geq 0$, $x\in\textbf{R}$:
\begin{eqnarray}
H_{\nu }(\xi (x),t) &=&\sum_{k}b_{k}H_{\nu }\left( \varphi \left( x-k\right)
x^{\nu },t\right)  \label{hnu} \\
\ &=&\sum_{k}b_{k}H_{\nu }\left( h\left( k+1-x\right) x^{\nu },t\right)
-\sum_{k}b_{k}H_{\nu }\left( h\left( k-x\right) x^{\nu },t\right)  \nonumber
\\ \nonumber
\ &=&\left[ \sum_{k}b_{k}\left( k+1\right) ^{\nu +1}J_{\nu +1}\left( t\left(
k+1\right) \right) -\sum_{k}b_{k}k^{\nu +1}J_{\nu +1}\left( tk\right)
\right] \\ \nonumber
&\times &H_{0}\left( \frac{1}{x},t\right).
\end{eqnarray}
All series converge in $L^{2}\left( {\bf R}\right)$

To write the above expression in a more compact form we use the addition
formula for Bessel functions
\[
J_{n}\left( x+y\right) =\sum_{k=-\infty }^{\infty }J_{k}\left( x\right)
J_{n-k}\left( y\right) .
\]
Then we get
\begin{eqnarray}\nonumber
H_{\nu }\left( \xi \left( x\right) ,t\right)  &=&\sum_{k}b_{k}\left[ \left(
k+1\right) ^{\nu +1}J_{\nu +1}\left( kt+t\right) -k^{\nu }J_{\nu +1}\left(
kt\right) \right] H_{0}\left( \frac{1}{x},t\right)   \\ \nonumber
&=&\sum_{k}b_{k}\left[ \left( k+1\right) ^{\nu +1}\sum_{h}J_{h}\left(
tk\right) J_{\nu +1-h}\left( t\right) -k^{\nu }J_{\nu +1}\left( kt\right) %
\right] \\ \label{hnu0}
&\times &H_{0}\left( \frac{1}{x},t\right).
\end{eqnarray}%
Define
\begin{equation}
m_{0}\left( t\right) =\sum_{k}b_{k}\left[ \left( k+1\right) ^{\nu
+1}\sum_{h}J_{h}\left( tk\right) J_{\nu +1-h}\left( t\right) -k^{\nu }J_{\nu
+1}\left( kt\right) \right] .  \label{m0}
\end{equation}%
Here we consider $L^{2}\left( {\bf R},\mu \right) $ with $d\mu
(x)=x\,dx$. By using the Plancherel Theorem, and the orthogonality of the
Haar functions, we get
\begin{eqnarray}
\frac{\delta _{k,0}}{2\left( \nu +1\right) } &=&\int_{0}^{\infty }\phi _{\nu
}^{\left( k\right) }\left( x\right) \phi _{\nu }^{\left( 0\right) }\left(
x\right) x\,dx  \label{risul0} \\
&=&\int_{0}^{\infty }\left[ H_{\nu }\left( x^{\nu }\left[ h\left(
k+1-x\right) -h\left( k-x\right) \right] ,t\right) \right. \\ \nonumber
&\times &\left. H_{\nu }\left( x^{\nu }%
\left[ h\left( 1-x\right) -h\left( -x\right) \right] ,t\right) \right] t\,dt
\nonumber \\ \nonumber
&=&\sum_{j\in {\bf Z}}\int_{j}^{j+1}[ H_{\nu }( x^{\nu }[
h( k+1-x) -h( k-x) ] ,t) \\ \nonumber
&\times & H_{\nu }(
x^{\nu }[ h( 1-x) -h( -x) ] ,t)
] t\,dt.
\end{eqnarray}%
Thus the latter, upon a change of variables, can be rewritten as
\begin{equation}
\int_{0}^{1}\sum_{j\in {\bf Z}}[ H_{\nu }( x^{\nu }[ h(
k+1-x) -h( k-x) ] ,t+j) H_{\nu }( x^{\nu }
[ h( 1-x) -h( -x) ] ,t+j) ]
( t+j) \,dt.  \label{risul1}
\end{equation}%
We used the following obvious fact:
\begin{equation}
\frac{1}{2\left( \nu +1\right) }=\int_{0}^{1}\frac{dt}{2\left( \nu +1\right)
}.  \label{valuedelta}
\end{equation}%
On comparing (\ref{valuedelta}) and (\ref{risul1}) for $k=0$, we get
\[
\sum_{j\in {\bf Z}}H_{\nu }^{2}\left( x^{\nu }\left[ h\left( 1-x\right)
-h\left( -x\right) \right] ,t+j\right) -\frac{1}{2\left( \nu +1\right) }%
=0,\qquad \rm {Lebesgue \ a.e.}
\]
On the other hand, in view of (\ref{risul0}), (\ref{m0}), (\ref{hnu0}) and (%
\ref{hnu}) the left-hand side of this equality can be rewritten in terms of $%
m_{0}$ as follows
\[
\sum_{j\in {\bf Z}}\left| m_{0}\left( t+j\right) \right| ^{2}\left|
H_{0}\left( 1/z,
t+j\right) \right| ^{2}=\frac{1}{2\left( \nu +1\right) },\qquad
\rm {Lebesgue \ a.e.}
\]

To get a direct connection with representations of $O_{\nu +1}$, we need to
consider
our new
multiresolutions on the complex plane ${\bf C}$. Assume  $\phi $ to
be a step function on ${\bf C}$, defined
for $\left| z\right| \leq 1$
by
\[
\phi \left( \left| z\right| e^{i%
\mathop{\rm Arg}%
\left( z\right) }\right) =\left\{
\begin{array}{lll}
1 &  & \text{if }0\leq
\mathop{\rm Arg}%
\left( z\right) \leq \alpha , \\
0 &  & \text{otherwise,}%
\end{array}%
\right.
\]%
where $\alpha =\frac{2\pi }{m}$, for a fixed $m\in\textbf{N}$. With $k,m\in
{\bf N}$, $1\leq N\leq m$ take then $V_{0}$ to be the span of
$\left\{ \phi \left[ \left( \left| z\right| +k\right) \exp \left( i\left(
\mathop{\rm Arg}%
\left( z\right) +N\alpha \right) \right) \right] \right\} $. Let
\begin{equation}
U\xi \left( z\right) =
\left( \nu +1\right) ^{-1/2}
\xi \left( \frac{z}{\nu +1}\right) \label{scalingoperator}
\end{equation}
be the scaling operator. For $j\in\textbf{Z}$ let $V_{j}$ be the closed span in
$L^{2}\left( {\bf C},\nu \right)$  of
\[
\left\{ \phi \left[ \left( \frac{\left| z\right| }{\left( \nu +1\right) ^{j}}%
+k\right) e^{i\left(
\mathop{\rm Arg}%
\left( z\right) +N\alpha \right) }\right] \right\} _{k\in {\bf Z},\;1\leq
N\leq m},
\]%

Consider $L^{2}\left( {\bf C},\nu \right) $ where the measure $\nu \left(
z\right) =z^{\nu }\,dz$, and $dz$ denotes the planar measure on ${\bf C}$.
Assume $U\phi \in V_{0}$, i.e.,
\[
\left( U\phi \right) \left( z\right) =\sum_{k}a_{k}\phi \left[ \left( \left|
z\right| +k\right) \exp \left( i\left(
\mathop{\rm Arg}%
\left( z\right) +N\alpha \right) \right) \right] , k\in\textbf{Z}
\]

\begin{proposition}
\label{ProHan.1}With the assumptions above, 
 the properties {\rm i)--iv)} of a multiresolution are
satisfied.
\end{proposition}

\begin{proof}
i) follows from the fact that the $\phi $'s have disjoint support
on $L^{2}\left( {\bf C},\nu \right) $. ii) holds
for Haar functions
and iii)
follows from i). By the density of step functions in $L^{2}\left( {\bf C}%
,\nu \right) $ also iv) follows.

If $\xi \in V_{-j}$ and $\nu\neq {-1}$, then $U^{j}\xi \in V_{0}$. Since
\[
\phi \in V_{0}\subset V_{-1}\text{\quad and\quad }\left\{ \phi \left[ \left(
\frac{\left| z\right| }{\nu +1}+k\right) \exp \left( i\left(
\mathop{\rm Arg}%
\left( z\right) +N\alpha \right) \right) \right] \right\}
\]%
are orthonormal in $V_{-1}$, we have
\[
\phi (z)=\sum_{k}a_{k}\phi \left[ \left( \frac{\left| z\right| }{\nu +1}%
+k\right) \exp \left( i\left(
\mathop{\rm Arg}%
\left( z\right) +N\alpha \right) \right) \right] , z\in \textbf {C},
\]%
so by applying the Hankel transform of order $\nu $ , we get
\[
H_{\nu }\left( \phi \left[ \left( 
\frac{\left| z\right| }{\nu +1}+k\right) \exp \left( i\left(
\mathop{\rm Arg}%
\left( z\right) +N\alpha \right) \right) \right] ;t\right) =m_{0}\left(
t\right) H_{0}\left( \frac{1}{z};t\right) . 
\]%
Using the orthogonality of $\phi \left[ \left( 
\left| z\right| +k\right) 
\exp \left( i\left( 
\mathop{\rm Arg}%
\left( z\right) +N\alpha \right) \right) \right] _{k\in {\bf Z}}$ in $L^{2}\left( {\bf C}%
,\nu \right) $we have
\begin{eqnarray}
\langle \phi ^{\left( k,N\right) }\mid\phi ^{\left( 0,0\right) }
\rangle &\equiv &\int \mkern-9mu\int_{{\bf C}}\phi ^{\left( k,N\right)
}\left( z\right) \overline{\phi ^{\left( 0,0\right) }\left( z\right) }%
z\,d\mu \left( z\right)  \label{inner-prod} \\
\ &=&\int_{0}^{\infty }\int_{0}^{2\pi }\phi \left[ \left( \left| z\right|
+k\right) \exp \left( i\left(
\mathop{\rm Arg}%
\left( z\right) +N\alpha \right) \right) \right] \nonumber \\
&\times &\overline{\phi \left[
\left| z\right| \exp \left( i%
\mathop{\rm Arg}%
\left( z\right) \right) \right] }  \nonumber \\
&\times &\left| z\right| ^{\nu +1}\exp \left( i%
\mathop{\rm Arg}%
\left( z\right) \left( \nu +1\right) \right) \,d\left| z\right| \,d%
\mathop{\rm Arg}%
\left( z\right)  \nonumber \\
\ &=&\int_{0}^{1}\left| z\right| ^{\nu +1}\delta _{k,0}\,d\left| z\right|
\int_{0}^{\alpha }\exp \left( i%
\mathop{\rm Arg}%
\left( z\right) \left( \nu +1\right) \right) \nonumber \\
&\times &\delta _{N,0}\,d%
\mathop{\rm Arg}%
\left( z\right)  \nonumber \\
\ &=&\frac{1}{\nu +2}\delta _{k,0}\frac{e^{i\alpha \left( \nu +1\right) }-1}{%
i\left( \nu +1\right) }\delta _{N,0}.  \nonumber
\end{eqnarray}%
By the Plancherel theorem, we then have
\[
\frac{1}{\nu +2}\delta _{k,0}\frac{e^{i\alpha \left( \nu +1\right) }-1}{%
i\left( \nu +1\right) }\delta _{N,0}=\int \mkern-9mu\int_{{\bf C}}H_{\nu
}\left( \phi ^{\left( k,N\right) }\left( z\right) ;t\right) \overline{H_{\nu
}\left( \phi ^{\left( 0,0\right) }\left( z\right) ;t\right) }t\,d\nu \left(
t\right) .
\]%
The left-hand side can then be rewritten as
\[
\int_{0}^{1}\int_{0}^{\alpha }H_{\nu }\left( \phi ^{\left( k,N\right)
}\left( z\right) ;t\right) \overline{H_{\nu }\left( \phi ^{\left( 0,0\right)
}\left( z\right) ;t\right) }t^{\nu +1}\,dt\,d%
\mathop{\rm Arg}%
\left( t\right) .
\]%
Upon a change of variable letting $\theta =%
\mathop{\rm Arg}%
\left( t+2\pi j\right) $, the latter equals
\[
\int_{0}^{1}\left| t\right| ^{\nu +1}\,d\left| t\right| \int_{0}^{\alpha
}e^{i2\pi \left( \nu +1\right) }\sum_{j}H_{\nu }\left( \phi ^{\left(
k,N\right) }\left( z\right) ;\left| t\right| e^{i\theta }\right) \overline{%
H_{\nu }\left( \phi ^{\left( 0,0\right) }\left( z\right) ;\left| t\right|
e^{i\theta }\right) }\,d\theta .
\]%
Comparing the previous two formulae for $k=N=0$ we get
\[
\sum_{j}\left| H_{\nu }\left( \phi ^{\left( 0,0\right) };\left| t\right|
e^{i\theta }\right) \right| ^{2}-\frac{1}{\nu +2}\frac{e^{i\alpha \left( \nu
+1\right) }-1}{i\left( \nu +1\right) }=0\qquad \rm {for \ Lebesgue \ a.e.}
\]%
Rewriting the above in terms of $m_{0}$ we have
\[
\sum_{j}\left| m_{0}\left( te^{2\pi ij}\right) \right| ^{2}\left|
H_{0}\left( \frac{1}{z};\left| t\right| e^{i\theta }\right)
\right| ^{2}=\frac{1}{\nu +2}\frac{e^{i\alpha \left( \nu +1\right) }-1}{%
i\left( \nu +1\right) },
\]%
since
\begin{eqnarray*}
&&\int_{0}^{1}\left| t\right| ^{\nu +1}\,d\left| t\right| \int_{0}^{\alpha
}e^{i2\pi \left( \nu +1\right) }\sum_{j}\left| m_{0}\left( te^{2\pi
ij}\right) \right| ^{2}\left| H_{0}\left( \frac{1}{z};\left|
t\right| e^{i\theta }\right) \right| ^{2}\,d\theta \\
&&\qquad =\int_{0}^{1}\left| t\right| ^{\nu +1}\,d\left| t\right|
\int_{0}^{\alpha }e^{i2\pi \left( \nu +1\right) }\sum_{j}\left|
m_{0}\left( te^{2\pi ij}\right) \right| ^{2}\frac{1}{\left| t\right| ^{2}}%
\,d\theta \\
&&\qquad =\int_{0}^{1}\left| t\right| ^{\nu -1}\,d\left| t\right|
\int_{0}^{\alpha }e^{i2\pi \left( \nu +1\right) }\sum_{j}\left|
m_{0}\left( te^{2\pi ij}\right) \right| ^{2}\,d\theta .
\end{eqnarray*}%
{}From (\ref{inner-prod})\label{in} we get
\begin{eqnarray*}
\int_{0}^{1}\left| t\right| ^{\nu -1}\,d\left| t\right| \int_{0}^{\alpha
}e^{i2\pi \left( \nu +1\right) }\sum_{j}\left| m_{0}\left( te^{2\pi
ij}\right) \right| ^{2}\,d\theta &=&\frac{\delta _{k,0}}{\nu +2}\frac{%
e^{i\alpha \left( \nu +1\right) }-1}{i\left( \nu +1\right) }\delta _{N,0} \\
&=&\frac{\nu }{\nu +2}\int_{0}^{1}\left| t\right| ^{\nu -1}\,d\left|
t\right| \\
&\times &\int_{0}^{\alpha }e^{i\theta \left( \nu +1\right) }\,d\theta .
\end{eqnarray*}%
Thus
\begin{equation}
\left( \frac{1}{\nu +1}\right)
\sum_{j}\left| m_{0}\left( te^{2\pi ij/\left( \nu +1\right) }\right) \right|
^{2}=\frac{\nu }{\nu +2}.  \label{eqast1}
\end{equation}%
Set $c=\frac{\nu }{\nu +2}$; then
$\left( \frac{1}{c\left( \nu +1\right) }\right)
\sum_{j}\left| m_{0}\left( te^{2\pi ij/\left( \nu +1\right) }\right) \right|
^{2}=1$.
\end{proof}

 In fact,
as in Ref.\ \CITE{dau}, Thm.~5.1.1, we have proved a part of the following result.

\begin{theorem}
\label{Thm1}If the ladder of the closed subspaces
$\left\{ V_{j}\right\}_{j\in {\bf Z}}$ in $L^{2}\left( {\bf C},\nu\right) $ 
satisfies properties {\rm i)--iv),} then there exists
an associated orthonormal wavelet basis $\left\{ \psi _{jk}:j,k\in {\bf Z}\right\} $ for
$L^{2}\left( {\bf C},\nu\right) $  such that
\[
\left( U\phi \right) \left( z\right) =\sum_{k}a_{k}\phi \left[ \left( \left|
z\right| +k\right) \exp \left( i\left(
\mathop{\rm Arg}%
\left( z\right) +N\alpha \right) \right) \right]
\]%
holds. One possibility for construction of the wavelet corresponding to
$\phi $ is that
\[
H_{\nu }\left( \phi \left[ \left( \left| z\right| +k\right)
\exp \left( i\left(
\mathop{\rm Arg}%
\left( z\right) +N\alpha \right) \right) \right] ;
\left( \nu +1\right)
t\right) =m_{0}\left(
t\right) H_{0}\left( \frac{1}{z};
t\right)
\]%
be satisfied.
\end{theorem}

\begin{completionofproof}.
We observe that the Bessel functions have a ``multiplicative periodicity'' on
the unit circle in the following sense:
\[
J_{\nu }\left( ze^{\pi ik}\right) =e^{\pi ik\nu }J_{\nu }\left( z\right)
\]

{}From the above \textup{(\ref{eqast1}),} this implies that
\[
c^{-1}\sum_{j=0}^{\nu }\left| m_{0}\left( ze^{2\pi ij/\left( \nu +1\right)
}\right) \right| ^{2}=\left( \nu +1\right) .
\]

Given $m_{0}$ satisfying \textup{(\ref{eqast1})}
there exists $\left\{ m_{i},\;i=1,\dots,\nu
\right\} $ from Corollary 4.2 of Ref.\ \CITE{Br-Jo2} such that
\[
\sum_{j=0}^{\nu}
c^{-1}
\overline{m_{k}\left( z\exp \left( 2\pi ij/\left( \nu +1\right) \right)
\right) }m_{k^{\prime }}\left( z\exp \left( 2\pi ij/\left( \nu +1\right)
\right) \right) =\delta _{kk^{\prime }}
\left( \nu +1\right) .
\]
Thus, reformulating
the orthogonality conditions in $L^{2}\left( {\bf C},\nu\right) $,
we get that the following matrix,
\[
M\left( z\right) =\frac{1}{\sqrt{c\left( \nu +1\right) }}\left(
\begin{array}{cccc}
m_{0}\left( \sigma _{0}\left( z\right) \right) &
m_{0}\left( \sigma _{1}\left( z\right) \right) & \dots &
m_{0}\left( \sigma _{\nu }\left( z\right) \right) \\
m_{1}\left( \sigma _{0}\left( z\right) \right) &
m_{1}\left( \sigma _{1}\left( z\right) \right) & \dots &
m_{1}\left( \sigma _{\nu }\left( z\right) \right) \\
\vdots & \vdots & \ddots & \vdots \\
m_{\nu }\left( \sigma _{0}\left( z\right) \right) & m_{\nu
}\left( \sigma _{1}\left( z\right) \right) & \dots & m_{\nu
}\left( \sigma _{\nu }\left( z\right) \right)%
\end{array}
\right) ,
\]
is unitary for Lebesgue almost all $z\in {\bf C}$.

Let $O_{\nu +1}$ be the $C^{\ast }$-algebra generated by $\nu +1$
isometries $S_{0},S_{1},\dots ,S_{\nu }$, $\nu\in \textbf{N }$ satisfying:
\[
S_{i}^{\ast }S_{j} =\delta _{i,j}1 ,\qquad
\sum_{i=0}^{\nu }S_{i}S_{i}^{\ast } =1.
\]

The representations we consider are
now
realized on the Hilbert space $%
H=L^{2}\left( {\bf C},\nu \right) $ where the measure
$\nu $
is given by $d\nu
\left( z\right) =z^{\nu }\,dz$.

As in Ref.\ \CITE{BJ}
the representation of the Cuntz algebra is defined in terms
of certain maps
\[
\sigma _{i}\colon \Omega \longrightarrow \Omega ,
\]%

such that $\mu \left( \sigma _{i}\left( \Omega \right) \cap \sigma
_{j}\left( \Omega \right) \right) =0$ for $i\neq j$, as in (\ref{eqCunpound.3}),
 and of measurable
functions $m_{0},\dots ,m_{\nu }\colon {\bf C\longrightarrow C}$. Also we have, for $L^{2}\left( {\bf C},\nu
\right) $:
\begin{equation}
\int_{{\bf C}}f\left( z\right) \,d\nu \left( z\right) =\sum_{r\in {\bf Z}%
_{\nu +1}}\rho _{r}\int_{{\bf C}}f\left( \sigma _{r}\left( z\right) \right)
\,d\nu \left( z\right) ,  \label{nuova}
\end{equation}%
where $\left\{ \rho _{r}\right\} $ is a (finite) probability distribution on
the cyclic group ${\bf Z}_{\nu +1}$.

The representations take the following form on $L^{2}\left( {\bf C},\nu
\right) $%
\[
\left( S_{k}\xi \right) \left( z\right) =m_{k}\left( z\right) \xi \left(
z^{\nu +1}\right) , \xi\in L^{2}\left( {\bf C},\nu
\right) :
\]%
where the functions $m_{k}$ are obtained from the above multiresolution
construction. It is easy to verify that $S_{k}$ is a representation of $O_{\nu
+1} $, $z\in\textbf{C}$ and that
\[
\left( S_{k}^{\ast }\xi \right) \left( z\right) =\sum_{r\in {\bf Z}_{\nu
+1}}
c^{-1}
\rho _{r}\overline{m_{k}\left( \sigma _{r}\left( z\right) \right) }\xi
\left( \sigma _{r}\left( z\right) \right) .
\]

In fact we have
\begin{eqnarray*}
\left( S_{k}^{\ast }S_{k^{\prime }}\xi \right) \left( z\right) &=&\sum_{r\in
{\bf Z}_{\nu +1}}
c^{-1}
\rho _{r}\overline{m_{k}\left( \sigma _{r}\left( z\right)
\right) }m_{k^{\prime }}\left( \sigma _{r}\left( z\right)
\right) \xi \left( \sigma \sigma _{r}\left( z\right) \right)
\\
&=&\delta _{k,k^{\prime }}\xi \left( z\right) ,
\end{eqnarray*}%
by the unitarity of the matrix $M\left( z\right) $. Similarly we may verify
that
\[
\sum_{k\in {\bf Z}_{\nu +1}}\left( S_{k}S_{k}^{\ast }\xi \right) \left(
z\right) =\xi \left( z\right), \xi\in L^{2}\left( {\bf C},\nu
\right) 
\]%
As a result, we then have indeed a representation of $O_{\nu +1}$.
\end{completionofproof}

\section{\label{KEW}A $\kew$-parametric construction of $\emm_{0}$}

Let us now turn to a $q$-parametric construction of $m_{0}$. We start by
giving a $q$-extension of the Hankel Fourier-Bessel integral. We use the
orthogonality relations from the following result (Theorem 3.10, p.~35 of %
Ref.\ \CITE{RS}) \rm{and} \CITE{Ga-Ra}.

\begin{theorem}
For $x\in\textbf{C}$ and $\left| x\right| <q^{-\frac{1}{2}}$, $n,m\in {\bf Z}$, $0<q<1$,we have
\begin{eqnarray*}
\delta _{m,n} &=&\sum_{k=-\infty }^{\infty }x^{k+n}q^{\frac{1}{2}\left(
k+n\right) }\frac{\left( x^{2}q;q\right) _{\infty }}{\left( q;q\right)
_{\infty }}\Phi _{1,1}\left( \left.
\begin{array}{c}
0 \\
x^{2}q%
\end{array}%
\right| q,q^{n+k+1}\right) \\
&&\qquad \times x^{k+m}q^{\frac{1}{2}\left( k+m\right) }\frac{\left(
x^{2}q;q\right) _{\infty }}{\left( q;q\right) _{\infty }}\Phi _{1,1}\left(
\left.
\begin{array}{c}
0 \\
x^{2}q%
\end{array}%
\right| q,q^{m+k+1}\right)
\end{eqnarray*}%
where the sum is absolutely convergent, uniformly on
compact subsets of the open disk $\left| x\right| <q^{-1/2}$.
\end{theorem}

We prove that the orthogonality relation of the
above theorem is a $q$-analogue of the Hankel
Fourier-Bessel integral (\ref{eA}).
To simplify notations, we replace $q$ by $q^{2}$ and $x$ by $q^{\alpha }$. For $%
\mathop{\rm Re}%
\left( \alpha \right) >-1$ this gives
\begin{eqnarray}
\delta _{m,n} &=&\sum_{k=-\infty }^{\infty }q^{\left( \alpha +1\right)
\left( k+n\right) }\frac{\left( q^{2\alpha +2};q^{2}\right) _{\infty }}{%
\left( q^{2};q^{2}\right) _{\infty }}\Phi _{1,1}\left( \left.
\begin{array}{c}
0 \\
q^{2\alpha +2}%
\end{array}%
\right| q^{2},q^{2n+2k+2}\right)   \label{e1.3} \\
&&\qquad \times q^{\left( \alpha +1\right) \left( k+n\right) }\frac{\left(
q^{2\alpha +2};q^{2}\right) _{\infty }}{\left( q^{2};q^{2}\right) _{\infty }}%
\Phi _{1,1}\left( \left.
\begin{array}{c}
0 \\
q^{2\alpha +2}%
\end{array}%
\right| q^{2},q^{2n+2k+2}\right) .  \nonumber
\end{eqnarray}%
Now rewrite (\ref{e1.3}) as the transform pair
\begin{eqnarray*}
g\left( q^{n}\right)  &=&\sum_{k=-\infty }^{\infty }q^{\left( \alpha
+1\right) \left( k+n\right) }\frac{\left( q^{2\alpha +2};q^{2}\right)
_{\infty }}{\left( q^{2};q^{2}\right) _{\infty }}\\
&\times &\Phi _{1,1}\left( \left.
\begin{array}{c}
0 \\
q^{2\alpha +2}%
\end{array}%
\right| q^{2},q^{2n+2k+2}\right) f\left( q^{k}\right) , \\
f\left( q^{k}\right)  &=&\sum_{k=-\infty }^{\infty }q^{\left( \alpha
+1\right) \left( k+n\right) }\frac{\left( q^{2\alpha +2};q^{2}\right)
_{\infty }}{\left( q^{2};q^{2}\right) _{\infty }}\\
& \times &\Phi _{1,1}\left( \left.
\begin{array}{c}
0 \\
q^{2\alpha +2}%
\end{array}%
\right| q^{2},q^{2n+2k+2}\right) g\left( q^{n}\right) ,
\end{eqnarray*}%
where $f,g$ are $L^{2}$-functions on the set $\left\{ q^{k}:k\in {\bf Z}%
\right\} $ with respect to the counting measure. Insert in the above formulae $J_{\alpha
}\left( x;q\right) $,
i.e., the $q$-Bessel function given by
\begin{eqnarray*}
J_{\alpha }\left( x;q\right)  &=&\frac{\left( q^{\alpha +1};q\right)
_{\infty }}{\left( q;q\right) _{\infty }}x^{\alpha }\sum_{k=0}^{\infty }%
\frac{\left( -1\right) ^{k}q%
{k+1 \choose 2}%
x^{2k}}{\left( q^{\alpha +1};q\right) _{k}\left( q;q\right) _{k}} \\
\  &=&\frac{\left( q^{\alpha +1};q\right) _{\infty }}{\left( q;q\right)
_{\infty }}x^{\alpha }\Phi _{1,1}\left( \left.
\begin{array}{c}
0 \\
q^{\alpha +1}%
\end{array}%
\right| q,x^{2}q\right), x\in \textbf{R\textit{}}
\end{eqnarray*}%
instead of $\frac{\left( q^{\alpha +1};q\right) _{\infty }}{\left( q;q\right)
_{\infty }}x^{\alpha }\Phi _{1,1}$ and replace $f\left( q^{k}\right) $ and $g\left( q^{n}\right)$ respectively
by $q^{k}f\left( q^{k}\right)$ and $q^{n}g\left( q^{n}\right)$.
This implies that $xf\left( x\right) $ and $xg\left( x\right) $ have to be $%
L^{2}$-functions respect to the $d_q$ measure on the set $\left\{ q^{k}:k\in {\bf Z}\right\} $, see $\CITE{Ga-Ra}$.
Hence we
have
\begin{eqnarray}
g\left( q^{n}\right)  &=&\sum_{k=-\infty }^{\infty }q^{2k}J_{\alpha }\left(
q^{k+n};q^{2}\right) f\left( q^{k}\right) ,  \label{e1.4} \\
f\left( q^{k}\right)  &=&\sum_{n=-\infty }^{\infty }q^{2n}J_{\alpha }\left(
q^{k+n};q^{2}\right) g\left( q^{n}\right) ,  \nonumber
\end{eqnarray}%
and the result follows.

\begin{remark}
When $q\longrightarrow 1$ with the condition
\[
\frac{\log \left( 1-q\right) }{\log q}\in 2{\bf Z},
\]%
we can replace $q^{k}$ and $q^{n}$ in {\rm (\ref{e1.4})} by $\left(
1-q\right) ^{\frac{1}{2}}q^{k}$ and $\left( 1-q\right) ^{\frac{1}{2}}q^{n}$
respectively. By using the following $q$-integral notation, {\rm {\CITE{Ga-Ra}, \ \cite{RS}}}
\begin{equation}
\int_{0}^{\infty }f\left( t\right) \,d_{q}t=\left( 1-q\right)
\sum_{k=-\infty }^{\infty }f\left( q^{k}\right) q^{k},  \label{q-measure}
\end{equation}%
then {\rm (\ref{e1.4})} takes the form
\begin{eqnarray*}
g\left( \lambda \right)  &=&\int_{0}^{\infty }f\left( x\right) J_{\alpha
}\left( \left( 1-q\right) \lambda x;q^{2}\right) x\,d_{q}\left( x\right) ,\\
f\left( x\right)  &=&\int_{0}^{\infty }g\left( \lambda \right) J_{\alpha
}\left( \left( 1-q\right) \lambda x;q^{2}\right) \lambda \,d_{q}\left(
\lambda \right) ,
\end{eqnarray*}%
where $\lambda $ in the first identity, and $x$ in the second identity, take
the values $q^{n}$, $n\in {\bf Z}$. For $q\longrightarrow 1$ we
therefore
obtain, at
least formally, the Hankel transform pair
\begin{eqnarray*}
g\left( \lambda \right)  &=&\int_{0}^{\infty }f\left( x\right) J_{\alpha
}\left( \lambda x\right) x\,dx, \\
f\left( x\right)  &=&\int_{0}^{\infty }g\left( \lambda \right) J_{\alpha
}\left( \lambda x\right) \lambda \,d\lambda .
\end{eqnarray*}
\end{remark}

We construct a $q$-analogue of a multiresolution via $q$-Hankel transforms.
To achieve that, let us proceed as we did in the previous section; but now we
replace the Hankel transform by the deformed one using a $q$-measure. Let us
consider as before the space $L^{2}\left( {\bf C}\right) $, but with the
measure $d\nu \left( z\right) $ replaced by the $q$-measure $d\nu_{q}\left(
z\right) $, i.e., $d\mu _{q}\left( z\right) =z^{\nu }\,d_{q}\left( z\right)
$, see \cite{RS},\cite{Ga-Ra}. Assume $\phi $ to be the function on ${\bf C}$ defined
for $\left| z\right| \leq 1$
by
\[
\phi \left( \left| z\right| e^{i%
\mathop{\rm Arg}%
\left( z\right) }\right) =\left\{
\begin{array}{lll}
1 &  & \text{if }0\leq
\mathop{\rm Arg}%
\left( z\right) \leq \alpha , \\
0 &  & \text{otherwise,}%
\end{array}%
\right.
\]%
where $\alpha =\frac{2\pi }{m}$, $m\in\textbf{N}$. Take then $V_{0}$ to be the
closed
span in $L^{2}\left( {\bf C},\nu _{q}\left( z\right)
\right) $ of $%
\left\{ \phi \left[ \left( \left| z\right| +k\right) \exp \left( i\left(
\mathop{\rm Arg}%
\left( z\right) +N\alpha \right) \right) \right] \right\} $, with $k,m\in
{\bf Z}$, $1\leq N\leq m$. Let
$U$
be the scaling operator (\ref{scalingoperator}). Let
\[
V_{j}=%
\mathop{\rm span}%
\left\{ \phi \left[ \left( \frac{\left| z\right| }{\left( \nu +1\right) ^{j}}%
+k\right) e^{i\left(
\mathop{\rm Arg}%
\left( z\right) +N\alpha \right) }\right] \right\} _{j\in {\bf Z},\;1\leq
N\leq m}.
\]

Let $\xi $ be a function on $L^{2}\left( {\bf C},\nu _{q}\left( z\right)
\right) $ given by
\[
\xi \left( z\right) =\sum_{k}a_{k}\left\{ \phi \left[ \left( \left| z\right|
+k\right) \exp \left( i\left(
\mathop{\rm Arg}%
\left( z\right) +N\alpha \right) \right) \right] \right\} .
\]

Assume $U\phi \in V_{0}$, i.e.,
\[
\left( U\phi \right) \left( z\right) =\sum_{k}a_{k}\phi \left[ \left( \left|
z\right| +k\right) \exp \left( i\left(
\mathop{\rm Arg}%
\left( z\right) +N\alpha \right) \right) \right] .
\]

\begin{proposition}
\label{ProKew.5}With the assumptions above, we conclude
that the properties {\rm i)--iv)} of a multiresolution are
satisfied.
\end{proposition}

\begin{proof}
i) follows from the fact that the $\phi $'s
have disjoint support on $L^{2}\left( {\bf C},\nu _{q}\right) $. ii) holds
as before
and iii) follows from i). By the density of step functions on $%
L^{2}\left( {\bf C},d\nu _{q}\right) $ also iv) follows.

If $\xi \in V_{-j}$ then $U^{j}\xi \in V_{0}$. By applying the $q$-Hankel
transform of order $\nu $ we get
\[
H_{\nu }^{q}\left( \phi \left[ \left( \left| z\right| +k\right)
\exp \left( i\left(
\mathop{\rm Arg}%
\left( z\right) +N\alpha \right) \right) \right] ;t\right) =m_{0}\left(
t\right) H_{0}^{q}\left( \frac{1}{z};t\right) ,
\]%
where we denote by $H_{\nu }^{q}\left( z;t\right) $ the $q$-Hankel transform
to avoid confusion wth the usual non deformed transform. The Plancherel Theorem for Hankel transforms extends in
a natural way to the case of $q$-Hankel transforms where it takes the
following form:%
\[
\int \mkern-9mu\int_{{\bf C}}tF\left( t\right) G\left( t\right) \,d\nu
_{q}\left( t\right) =\int \mkern-9mu\int_{{\bf C}}zf\left( z\right) g\left(
z\right) \,d\nu _{q}\left( z\right) ,
\]%
where $d\nu _{q}\left( z\right)$ is the $q$-measure (\ref{q-measure}).

Then by using the orthogonality of $\phi \left[ \left( \left| z\right|
+k\right) \exp \left( i\left(
\mathop{\rm Arg}%
\left( z\right) +N\alpha \right) \right) \right] _{k\in {\bf Z}}$ and the
following fact:
\[
\int_{0}^{1}\left| z\right| ^{\nu +1}\,d_{q}\left| z\right| =\frac{1-q}{%
1-q^{\nu +2}}, 0<q<1
\]%
we have
\begin{eqnarray}
\langle \phi ^{\left( k,N\right) }\mid\phi ^{\left( 0,0\right) }\rangle &\equiv
&\int \mkern-9mu\int_{{\bf C}}\phi ^{\left( k,N\right) }\left( z\right)
\overline{\phi ^{\left( 0,0\right) }\left( z\right) }z\,d\mu _{q}\left(
z\right)  \label{inner-prod1} \\
\ &=&\int_{0}^{\infty }\int_{0}^{2\pi }\phi \left[ \left( \left| z\right|
+k\right) \exp \left( i\left(
\mathop{\rm Arg}%
\left( z\right) +N\alpha \right) \right) \right] \nonumber \\
&\times &\overline{\phi \left[
\left| z\right| \exp \left( i%
\mathop{\rm Arg}%
\left( z\right) \right) \right] }  \nonumber \\
&\times &\left| z\right| ^{\nu +1}\exp \left( i%
\mathop{\rm Arg}%
\left( z\right) \left( \nu +1\right) \right) \,d\left| z\right| \,d%
\mathop{\rm Arg}%
\left( z\right)  \nonumber \\
\ &=&\int_{0}^{1}\left| z\right| ^{\nu +1}\delta _{k,0}\,d\left| z\right|
\nonumber \\
&\times &\int_{0}^{\alpha }\exp \left( i%
\mathop{\rm Arg}%
\left( z\right) \left( \nu +1\right) \right) \delta _{N,0}\,d%
\mathop{\rm Arg}%
\left( z\right)  \nonumber \\
\ &=&\frac{1-q}{1-q^{\nu +2}}\delta _{k,0}\frac{e^{i\alpha \left( \nu
+1\right) }-1}{i\left( \nu +1\right) }\delta _{N,0}.  \nonumber
\end{eqnarray}
By the Plancherel theorem, we then have
\begin{eqnarray*}
\frac{1-q}{1-q^{\nu +2}}\delta _{k,0}\frac{e^{i\alpha \left( \nu +1\right)
}-1}{i\left( \nu +1\right) }\delta _{N,0}\\
=\int \mkern-9mu\int_{{\bf C}%
}H_{\nu }^{q}\left( \phi ^{\left( k,N\right) }\left( z\right) ;t\right)
\overline{H_{\nu }^{q}\left( \phi ^{\left( 0,0\right) }\left( z\right)
;t\right) }t\,d\nu _{q}\left( t\right) .
\end{eqnarray*}
The left-hand side can then be rewritten as
\[
\int_{0}^{1}\int_{0}^{\alpha }H_{\nu }^{q}\left( \phi ^{\left( k,N\right)
}\left( z\right) ;t\right) \overline{H_{\nu }^{q}\left( \phi ^{\left(
0,0\right) }\left( z\right) ;t\right) }t^{\nu +1}\,d_{q}t\,d%
\mathop{\rm Arg}%
\left( t\right) .
\]%
Upon a change of variable setting $\theta =%
\mathop{\rm Arg}%
\left( t+2\pi j\right) $,$j\in\textbf{Z}$ the latter expression is equal
\[
\int_{0}^{1}\left| t\right| ^{\nu +1}\,d_{q}\left| t\right| \int_{0}^{\alpha
}e^{i\theta \left( \nu +1\right) }\sum_{j}H_{\nu }^{q}\left( \phi ^{\left(
k,N\right) }\left( z\right) ;\left| t\right| e^{i\theta }\right) \overline{%
H_{\nu }^{q}\left( \phi ^{\left( 0,0\right) }\left( z\right) ;\left|
t\right| e^{i\theta }\right) }\,d\theta .
\]%
Comparing the previous two formulae for $k=N=0$ we get
\[
\sum_{j}\left| H_{\nu }^{q}\left( \phi ^{\left( 0,0\right) };\left| t\right|
e^{i\theta }\right) \right| ^{2}-\frac{1-q}{1-q^{\nu +2}}\frac{e^{i\alpha
\left( \nu +1\right) }-1}{i\left( \nu +1\right) }=0\qquad {\nu_q -\ \rm a.e.}
\]%
Rewriting the above in terms of $m_{0}$ we have
\[
\sum_{j}\left| m_{0}\left( te^{2\pi ij}\right) \right| ^{2}\left|
H_{0}^{q}\left( \frac{1}{z};\left| t\right| e^{i\theta
}\right) \right| ^{2}=\frac{1-q}{1-q^{\nu +2}}\frac{e^{i\alpha \left( \nu
+1\right) }-1}{i\left( \nu +1\right) }.
\]
From (\ref{inner-prod1}) we have
\begin{eqnarray*}
\int_{0}^{1}\left| t\right| ^{\nu -1}\,d_{q}\left| t\right| \int_{0}^{\alpha
}e^{i\theta \left( \nu +1\right) }\sum_{j}\left| m_{0}\left( te^{2\pi
ij}\right) \right| ^{2}\,d\theta \\
=\frac{1-q}{1-q^{\nu +2}}\delta _{k,0}
\frac{e^{i\alpha \left( \nu +1\right) }-1}{i\left( \nu +1\right) }\delta
_{N,0} \\
=\frac{1-q^{\nu }}{1-q^{\nu +2}}\int_{0}^{1}\left| t\right| ^{\nu
-1}
\,d_{q}\left| t\right|
\int_{0}^{\alpha }e^{i\theta \left( \nu +1\right)
}\,d\theta ,
\end{eqnarray*}%
thus
\[
\left( \frac{1}{\nu +1}\right)
\sum_{j}\left| m_{0}\left( te^{2\pi ij/\left( \nu +1\right) }\right) \right|
^{2}=\frac{1-q^{\nu }}{1-q^{\nu +2}}.
\]
Set $c_{q}=\frac{1-q^{\nu }}{1-q^{\nu +2}}$; thus we get
$\frac{1}{c_{q}\left( \nu +1\right) }
\sum_{j}\left| m_{0}\left( te^{2\pi ij/\left( \nu +1\right) }\right) \right|
^{2}=1$.

We notice that the Bessel functions have a ``multiplicative periodicity'' on
the unit circle in the following sense:
\[
J_{\nu }\left( ze^{\pi ik}\right) =e^{\pi ik\nu }J_{\nu }\left( z\right) .
\]

This implies that
\[
c_{q}^{-1}\sum_{j=0}^{\nu }\left| m_{0}\left( te^{2\pi ij/\left( \nu +1\right)
}\right) \right| ^{2}=\left( \nu +1 \right) .
\]%
Thus a $q$-analogue of Theorem \ref{Thm1} holds. As in the previous section
we construct representations of the Cuntz algebra in terms of the functions $%
m_{i}$ whose existence is guaranteed from Corollary 4.2 of Ref.\
\CITE{Br-Jo2}.

As before we construct representations of the algebra $O_{\nu +1}$ associated
to the above multiresolution for the $q$-deformed case. The representations are
realized on a Hilbert space $H=L^{2}\left( {\bf C},d\nu _{q}\right) $ where
the measure is given by $d\nu _{q}\left( z\right) =z^{\nu }\,d_{q}z$.
\end{proof}

We now turn to the representation of the Cuntz algebra $O_{\nu +1}$. It is given in terms
of certain maps
\begin{equation}
\sigma _{k}\text{: }\Omega \longrightarrow \Omega ,\qquad \sigma _{k}\left(
z\right) =\sigma _{0}\left( z\right) e^{ik2\pi /\left( \nu +1\right) }, \nu\in\textbf{N}
\label{representationmaps}
\end{equation}
where
\[
\sigma _{0}\left( z\right) ^{\nu +1}=z, z\in\textbf{C},
\]%
such that $\mu _{q}\left( \sigma _{i}\left( \Omega \right) \cap \sigma
_{j}\left( \Omega \right) \right) =0$ for $i\neq j$. Also, for $f\in L^{2}\left( {\bf C},d\nu
_{q}\right) $
\begin{equation}
\int_{{\bf C}}f\left( z\right) \,d\nu _{q}\left( z\right) =\sum_{r\in {\bf Z}
_{\nu +1}}\rho _{r}\int_{{\bf C}}f\left( \sigma _{r}\left( z\right) \right)
\,d\nu _{q}\left( z\right) .  \label{nuova1}
\end{equation}%
In fact, we have $\mu _{q}\left( \sigma _{r}\left( E\right) \right) =\rho
_{r}\mu _{q}\left( E\right) $, for Borel subsets $E\subset {\bf C}$.

The representation takes the following form on $L^{2}\left( {\bf C},d\nu
_{q}\right) $:%
\[
\left( S_{k}\xi \right) \left( z\right) =m_{k}\left( z\right) \xi \left(
z^{\nu +1}\right) ,
\]%
where the functions $m_{k}$ are obtained from the above multiresolution
construction. Then we have
\[
\left( S_{k}^{\ast }\xi \right) \left( z\right) =\sum_{r\in {\bf Z}_{\nu
+1}}
c_{q}^{-1}
\rho _{r}\overline{m_{k}\left( \sigma _{r}\left( z\right) \right) }\xi
\left( \sigma _{r}\left( z\right) \right) .
\]

Thus:
\begin{eqnarray*}
\left( S_{k}^{\ast }S_{k^{\prime }}\xi \right) \left( z\right) &=&\sum_{r\in
{\bf Z}_{\nu +1}}
c_{q}^{-1}
\rho _{r}\overline{m_{k}\left( \sigma _{r}\left( z\right)
\right) }m_{k^{\prime }}\left( \sigma _{r}\left( z\right)
\right) \xi \left( \sigma \sigma _{r}\left( z\right) \right)
\\
&=&\delta _{k,k^{\prime }}\xi \left( z\right)
,
\end{eqnarray*}%
by the unitarity of the matrix $M\left( z\right) $.
We have used the convention
$\sigma \left( z\right) =z^{\nu +1}$
and the fact that
$\sigma \circ \sigma_{r}=\mathop{\rm id}$
for all $r$.
It is easy
similarly
to verify
that
\[
\sum_{k\in {\bf Z}_{\nu +1}}\left( S_{k}S_{k}^{\ast }\xi \right) \left(
z\right) =\xi \left( z\right) .
\]%
As a result, we then have a representation of $O_{\nu +1}$.

\section{\label{MUL}Multiresolution analysis}

We study now a particular case of a construction of a multiresolution. We
then see how to construct a representation of the Cuntz algebra. It is
interesting to see that for the corresponding representation so constructed
we get a
$q$-number related to the modulus of a Markov trace \cite{ChPr94}
for compact quantum groups of
type B.\cite{Pa-La}

Let us consider
\[
V_{0}=%
\mathop{\rm closed\;span}%
\left\{ \left[ h\left( q^{k}-z\right) -h\left( q^{k+1}-z\right) \right]
\right\} _{k\in {\bf Z}}, 0<q<1
\]%
Consider the step function given by
\begin{equation}
\phi _{\nu }^{\left( k\right) }\left( z,q\right) =\left[ h\left(
q^{k}-z\right) -h\left( q^{k+1}-z\right) \right] .  \label{phikq}
\end{equation}%
Set
\[
\left[ \nu +1\right] _{q^{2}}=\frac{1-q^{2\left( \nu +1\right) }}{1-q^{2}}
\]%
and define the scaling
\[
Uf\left( z\right) =\left( \nu +1\right) ^{-\frac{1}{2}}f\left( \left( \nu
+1\right) ^{-1}z\right) , f\in L^{2}\left( {\bf C},\nu _{q}\right)
\]%
Assume $U\phi _{\nu }^{\left( k\right) }\in V_{0}$, then
\[
U\phi _{\nu }^{\left( k\right) }\left( z,q\right) =\sum_{k}a_{k}\left[
h\left( q^{k}-z\right) -h\left( q^{k+1}-z\right) \right] .
\]%
It follows
\[
U^{j}\phi _{\nu }^{\left( k\right) }\left( z,q\right) =\sum_{k}a_{k}\left[
h\left( q^{k}-\frac{z}{\left( \nu +1\right) ^{j}}\right) -h\left( q^{k+1}-%
\frac{z}{\left( \nu +1\right) ^{j}}\right) \right] .
\]%
Let $V_{j}=U^{j}V_{0}$, so that if $f\in V_{j}$, $U^{-j}f\in V_{0}$.
The set $\left\{ \phi _{\nu }^{\left( k\right) }\right\} _{k\in Z}$ is an
orthonormal set in $L^{2}\left( {\bf C}\right) $. In fact
\[
h\left( q^{k}-z\right) -h\left( q^{k+1}-z\right) =\left\{
\begin{array}{lll}
1 &  & \text{if }q^{k+1}<\left| z\right| <q^{k}, \\
0 &  & \text{otherwise,}%
\end{array}%
\right.
\]%
are defined for $q^{k+1}<\left| z\right| <q^{k}$ in the annulus of $%
r=q^{k+1} $, $R=q^{k}$. It follows that the set (\ref{phikq}) $\left\{ \phi
_{\nu }^{\left( k\right) }\left( z,q\right) \right\} _{k\in {\bf Z}}$ is
orthogonal in $L^{2}\left( {\bf C}\right) $ since the functions $\phi _{\nu
}^{\left( k\right) }$ have disjoint support. Actually the set is orthogonal
in $L^{2}\left( {\bf T}\right) $ since for $k\longrightarrow \infty $, we have $%
q^{k}\longrightarrow 0$, and for $k\longrightarrow 0$, we have $q^{k}\longrightarrow
1$. Let
\begin{equation}
\xi \left( z\right) =\sum_{k}a_{k}\left[ h\left( q^{k}-z\right) -h\left(
q^{k+1}-z\right) \right] .  \label{asterisco}
\end{equation}%
By applying the $q$-Hankel transform
$\xi\rightarrow\hat\xi$ to both sides of (\ref{asterisco}) we
get then
\[
\hat{\xi}\left( t\right) =\sum_{k}a_{k}H_{\nu }^{q}\left[ \left( h\left(
q^{k}-z\right) -h\left( q^{k+1}-z\right) \right) ;t\right] ,
\]%
which implies
\begin{eqnarray*}
\hat{\xi}\left( t\right) = H_{0}^{q}\left( \frac{1}{z};t\right) \times\\
\left[ \sum_{k}a_{k}\left( q^{k\left( \nu
+1\right) }J_{\nu +1}\left( \left( 1-q\right) tq^{k};q\right) -q^{\left(
k+1\right) \left( \nu +1\right) }J_{\nu +1}\left( \left( 1-q\right)
tq^{k+1};q\right) \right) \right] \\
=\left[ \sum_{k}a_{k}q^{k\left( \nu +1\right) }\right] \left[ J_{\nu
+1}\left( \left( 1-q\right) tq^{k};q\right) -q^{\nu +1}J_{\nu +1}\left(
\left( 1-q\right) tq^{k+1};q\right) \right] \\
\times H_{0}^{q}\left( \frac{1}{z};t\right) .
\end{eqnarray*}

Using the Plancherel theorem for $q$-Hankel transforms
and orthogonality of $\left\{ \phi
_{\nu }^{\left( k\right) }\right\} _{k\in {\bf Z}}$ as before,
since we have
\[
\delta _{k,0}=1=\frac{1}{1-q}\int_{q}^{1}t\,d_{q}t,
\]%
the left-hand side becomes then

\begin{eqnarray*}
 0&=&\int_{q}^{1}\left[ \sum_{j\in {\bf Z}}q^{2j}H_{\nu }^{q}\left( \left(
h\left( q^{k}-z\right) -h\left( q^{k+1}-z\right) ;q^{j}s\right) \right) \right. \\
&\times &\overline{H_{\nu }^{q}\left( \left( h\left( 1-z\right) -h\left(
q^{1}-z\right) ;q^{j}s\right) \right) }
\left. -\frac{1}{1-q^{2\left( \nu +1\right) }}\right]
s\,d_{q}s,
\end{eqnarray*}
so that almost everywhere with respect to $d_q$,

\begin{eqnarray*}
 \label{hankel-1}
&&\sum_{j\in {\bf Z}}q^{2j}H_{\nu }^{q}\left( \left( h\left( q^{k}-z\right)
-h\left( q^{k+1}-z\right) ;q^{j}s\right) \right) \\
&\times &\overline{H_{\nu}^{q}\left( \left( h\left( 1-z\right) -h\left( q^{1}-z\right);q^{j}s\right)
\right)} =\frac{1}{1-q^{2\left( \nu +1\right) }}. 
\end{eqnarray*}%
Now we have by using the above 
\begin{eqnarray*}
&&\sum_{j\in {\bf Z}}q^{2j}H_{\nu }^{q}\left( \left( h\left( q^{k}-z\right)
-h\left( q^{k+1}-z\right) ;q^{j}t\right) \right) \\
&\times &\overline{H_{\nu}^{q}\left( \left( h\left( 1-z\right) -h\left( q^{1}-z\right)
;q^{j}t\right) \right) } =\sum_{j\in {\bf Z}}\left| m_{0}\left( tq^{j}\right) \right|
^{2}\left| H_{0}^{q}\left( \frac{1}{z};tq^{j}\right) \right| ^{2}.
\end{eqnarray*}%
Hence we have:%
\[
\sum_{j\in {\bf Z}}\left| m_{0}\left( tq^{j}\right) \right| ^{2}\left|
H_{0}^{q}\left( \frac{1}{z};tq^{j}\right) \right| ^{2}=\frac{1}{%
1-q^{2\left( \nu +1\right) }}.
\]

By a similar argument as above we get the special property for the function $%
m_{0}$:%
\begin{equation}
\sum_{j\in {\bf Z}}\left| m_{0}\left( tq^{j}\right) \right| ^{2}\left|
H_{0}^{q}\left( \frac{1}{z};tq^{j}\right) \right| ^{2}=\frac{1}{1-q^{2\left(
\nu +1\right) }}.  \label{sum-m0}
\end{equation}

Observe that in this case since $q\leq \left| t\right| \leq 1$ and then from
$q\leq q^{1-j}\leq \left| t\right| q^{-j}\leq q^{-j}\leq 1$ we have $\left|
t\right| \leq q^{j}\leq 1$ and then
\[
j\leq \frac{\log \left| t\right| }{\log q}.
\]%
For $\left| t\right| =q$, $j=1$ and for $\left| t\right| =1$, $j=0$. Hence
the sum in (\ref{sum-m0}) reduces to a finite sum, by using a similar
argument as for the Haar wavelet multiresolution.\cite{dau} For a scale $%
\nu +1$, $\nu\in\textbf{N}$ we thus have
\[
\sum_{j=0}^{\nu }\left| m_{0}\left( tq^{j}\right) \right| ^{2}\left|
H_{0}^{q}\left( \frac{1}{z};tq^{j}\right) \right| ^{2}=\frac{1}{%
1-q^{2\left( \nu +1\right) }}.
\]

In this case we should note that $\left| H_{0}\left( \frac{1}{z}%
;tq^{j}\right) \right| ^{2}=q^{-2j}$. Thus it follows:%
\[
\sum_{j=0}^{\nu }q^{-2j}\left| m_{0}\left( tq^{j}\right) \right| ^{2}=\frac{1
}{1-q^{2\left( \nu +1\right) }}.
\]
Set $d_{q}=\frac{1}{1-q^{2\left( \nu +1\right) }}$; then
$d_{q}^{-1}\sum_{j=0}^{\nu }q^{-2j}
\left| m_{0}\left( tq^{j}\right) \right| ^{2}=1$.

With the function $m_{0}$ given choose $m_{1},\dots ,m_{\nu }$ in $%
L^{2}\left( {\bf T,},\nu_{q}\right) $ such that
\begin{equation}
\sum_{j=0}^{\nu }q^{-2j}m_{r}\left( tq^{j}\right) \overline{m_{r^{\prime
}}\left( tq^{j}\right) }=\delta _{r,r^{\prime }}\frac{1}{1-q^{2\left( \nu
+1\right) }} .  \label{sum2}
\end{equation}

Define the functions $\psi _{1},\psi _{2},\dots ,\psi _{\nu }$ by the
formula:
\begin{equation}
H_{\nu }^{q}\left( \psi _{r}^{\left( j,m\right) }\left( z\right) ;t\left(
\nu +1\right) \right) =m_{r}\left( t\right) H_{0}^{q}\left( \frac{1}{z}%
;t\right) .  \label{fun1}
\end{equation}

Concretely the functions in (\ref{fun1}) are $\psi _{r}^{\left( j,m\right)
}\left( z\right) =\psi _{r}\left( \left( \nu +1\right) ^{-m}z-q^{j}\right) $%
. Then using (\ref{sum2}) and (\ref{fun1}) it follows that
\[
\left\{ \left( \nu +1\right) ^{\frac{-m}{2}}\psi _{r}^{\left( j,m\right)
}\left( z\right) \right\} _{j,m}
\]%
is an orthogonal basis for the space $V_{-1}\cap V_{0}^{\bot }$ and then by
iii) and iv) they form an orthogonal basis for $L^{2}\left( {\bf C},\nu
_{q}\right) $.

Now reformulating (\ref{sum2}), the orthonormality of $\left\{ \left( \nu
+1\right) ^{\frac{-m}{2}}\psi _{r}^{\left( j,m\right) }\left( z\right)
\right\} _{j,m}$ is equivalent to the following matrix $M\left( t\right)$ with entries,
\[
\frac{1}{\sqrt{d_{q}\left( \nu +1\right) }}
\left(
\begin{array}{cccc}
\sqrt{\rho _{0}}m_{0}\left( \sigma _{0}\left( t\right) \right) &
\sqrt{\rho _{1}}m_{0}\left( \sigma _{1}\left( t\right) \right) & \dots &
\sqrt{\rho _{\nu }}m_{0}\left( \sigma _{\nu }\left( t\right) \right) \\
\sqrt{\rho _{0}}m_{1}\left( \sigma _{0}\left( t\right) \right) &
\sqrt{\rho _{1}}m_{1}\left( \sigma _{1}\left( t\right) \right) & \dots &
\sqrt{\rho _{\nu }}m_{1}\left( \sigma _{\nu }\left( t\right) \right) \\
\vdots & \vdots & \ddots & \vdots \\
\sqrt{\rho _{0}}m_{\nu }\left( \sigma _{0}\left( t\right) \right) &
\sqrt{\rho _{1}}m_{\nu }\left( \sigma _{1}\left( t\right) \right) & \dots &
\sqrt{\rho _{\nu }}m_{\nu }\left( \sigma _{\nu }\left( t\right) \right)%
\end{array}%
\right) ,
\]%
being unitary, where $\rho _{j}=q^{-2j}$.

The class of representations of the algebra $O_{\nu +1}$ associated to the
above multiresolution construction is given as in the previous cases in
terms of the functions $m_{i}$ and of the maps $\sigma _{i}$. The
representations are realized on the Hilbert space $H=L^{2}\left( {\bf C}%
,\nu _{q}\right) $, where as before $d\nu _{q}\left( z\right) =z^{\nu
}\,d_{q}z.$ A similar construction works for the case $q=1$ where we use
classical Bessel functions and the usual Hankel transform.

Define the representation of the Cuntz algebra in terms of certain maps
(analogous to (\ref{representationmaps})):
\[
\sigma _{i}\colon \Omega \longrightarrow \Omega ,\qquad \sigma _{i}\left(
z\right) =\sigma _{0}\left( z\right) q^{i},
\]%
where
\[
\sigma _{0}\left( z\right) ^{\nu +1}=z,
\]%
such that $\mu _{q}\left( \sigma _{i}\left( \Omega \right) \cap \sigma
_{j}\left( \Omega \right) \right) =0$ for $i\neq j$.
Hence, the system (\ref{eqCunpound.3}) here will be the
$N$-sheeted Riemann surface of $\sqrt[N]{z}$.
Also for $L^{2}\left( {\bf C},d\nu
_{q}\right) $
\begin{equation}
\int_{{\bf C}}f\left( z\right) \,d\nu _{q}\left( z\right) =\sum_{r\in {\bf Z}
_{\nu +1}}\rho _{r}\int_{{\bf C}}f\left( \sigma _{r}\left( z\right) \right)
\,d\nu _{q}\left( z\right) ,  \label{nuova2}
\end{equation}
which is the analogue of (\ref{nuova}).
In fact, $\mu _{q}\left( \sigma _{r}\left( E\right) \right) =\rho _{r}\mu
_{q}\left( E\right) $ with $\rho _{i}=q^{-2i}$, for Borel sets $E\subset {\bf
C}$.

The representation takes the following form on $L^{2}\left( {\bf C},\nu
_{q}\right) $:%
\[
\left( S_{k}\xi \right) \left( z\right) =m_{k}\left( z\right) \xi \left(
z^{\nu +1}\right) ,
\]%
where the functions $m_{k}$ are obtained from the above multiresolution
construction. By using (\ref{nuova2}) we have
\[
\left( S_{k}^{\ast }\xi \right) \left( z\right) =\sum_{r\in {\bf Z}_{\nu
+1}}
d_{q}^{-1}
\rho _{r}\overline{m_{k}\left( \sigma _{r}\left( z\right) \right) }\xi
\left( \sigma _{r}\left( z\right) \right) .
\]

Thus:
\begin{eqnarray*}
\left( S_{k}^{\ast }S_{k^{\prime }}\xi \right) \left( z\right) &=&\sum_{r\in
{\bf Z}_{\nu +1}}
d_{q}^{-1}
\rho _{r}\overline{m_{k}\left( \sigma _{r}\left( z\right)
\right) }m_{k^{\prime }}\left( \sigma _{r}\left( z\right)
\right) \xi \left( \sigma \sigma _{r}\left( z\right) \right)
\\
&=&\delta _{k,k^{\prime }}\xi \left( z\right) ,
\end{eqnarray*}%
by the unitarity of the matrix $M\left( z\right) $. It is easy to verify
that
\[
\sum_{k\in {\bf Z}_{\nu +1}}\left( S_{k}S_{k}^{\ast }\xi \right) \left(
z\right) =\xi \left( z\right) .
\]

We then have a representation of $O_{\nu +1}$. The interesting feature in
this case is the fact that the $q$-number
\[
\frac{1}{1-q^{2\nu +2}}=\frac{1}{1-q^{2}}\left[ \nu +1\right] _{q^{2}}^{-1}
\]%
appearing in the orthogonality relations is exactly a multiple of the
modulus of the Markov trace\cite{ChPr94}
associated to the compact quantum group of type
B.

Then we can perform a Fourier-type analysis over the cyclic group $Z_{\nu +1}
$ introducing
\[
A_{i,j}\left( z\right) =
\left( \frac{1}{1-q^{2\left( \nu +1\right) }}\right) ^{-1}
\sum_{\omega
\colon \omega ^{\nu +1}=z}\omega ^{-j}m_{i}\left( \omega \right)
\]%
and the inverse transform
\[
m_{i}\left( z\right) =\sum_{j=0}^{\nu }z^{j}A_{i,j}\left( z^{\nu +1}\right) .
\]

\section{\label{TIG}Tight Frames, deformed Tight Frames and representations
of $O_{\protect\nu +1}$}

In this section we construct tight frames giving rise to certain
representations of the Cuntz algebra.

The representations we will consider are realized on a Hilbert space $%
H=L^2\left( \Omega ,\mu \right) $ where $\Omega $ is a measure space and $%
\mu $ is a probability measure on $\Omega $.

A frame is a set of non-independent vectors which can be used to construct
an explicit and complete expansion for every vector in the space. Thus we
have the following definition:

\begin{definition}
A family of functions $\left\{ \varphi _{j}\right\} _{j\in J}$ in a Hilbert
space $H$ is called a frame if there exist $0<A<\infty$, $0<\textit{}B<\infty $ so that for
all $f$ in $H$ we have:
\[
A\left\| f\right\| ^{2}\leq \sum_{j\in J}\left| \left\langle f\mid\varphi
_{j}\right\rangle \right| ^{2}\leq B\left\| f\right\| ^{2}.
\]
\end{definition}

We call $A$ and $B$ the frame bounds. If the two frame bounds are equal then
as in Ref.\ \CITE{dau} the frame will be called a tight frame. Thus in a tight
frame we have, for all $f\in H$,
\[
\text{ }\sum_{j\in J}\left| \left\langle f\mid\varphi _{j}\right\rangle \right|
^{2}=A\left\| f\right\| ^{2},
\]%
where $\left\langle f\mid\varphi _{j}\right\rangle $ are the Fourier
coefficients.

We construct tight frames but instead of a Fourier transform we use the
Hankel transform ( defined in the previous sections). We will see that the
construction will then extend to a $q$-deformed tight frame.

Let us start with functions $m_{0},m_{1},\dots ,m_{\nu }\colon {\bf %
T\rightarrow C}$ such that the following $\nu +1\times \nu +1$ matrix
\[
M\left( t\right) =
\frac{1}{\sqrt{\nu +1}}
\left(
\begin{array}{cccc}
m_{0}\left( \sigma _{0}\left( t\right) \right)  & m_{0}\left( \sigma
_{1}\left( t\right) \right)  & \dots  & m_{0}\left( \sigma _{\nu }\left(
t\right) \right)  \\
m_{1}\left( \sigma _{0}\left( t\right) \right)  & m_{1}\left( \sigma
_{1}\left( t\right) \right)  & \dots  & m_{1}\left( \sigma _{\nu }\left(
t\right) \right)  \\
\vdots  & \vdots  & \ddots  & \vdots  \\
m_{\nu }\left( \sigma _{0}\left( t\right) \right)  & m_{\nu }\left( \sigma
_{1}\left( t\right) \right)  & \dots  & m_{\nu }\left( \sigma _{\nu }\left(
t\right) \right)
\end{array}%
\right)
\]%
is unitary for almost all $z\in {\bf T}$. Assume that $m_{0}\left( 0\right)
=1$ and that the following infinite product:
\[
H_{k}\left( \varphi \left( z\right) ;t\right) =\prod_{l=1}^{\infty }
m_{0}\left( \left( \nu +1\right) ^{-l}t\right)
\]%
converges pointwise almost everywhere. By Ref.\ \CITE{dau} it follows from the
condition%
\[
\sum_{j=0}^{\nu }\left| m_{0}\left( te^{\pi ij}\right) \right| ^{2}=1
\]%
that $H_{k}\left( \varphi \left( z\right) ;t\right) \in L^{2}\left( {\bf T}%
\right) $, and that $\left\| \varphi \right\| _{2}\leq 1$. Let us now define
$\psi _{1},\psi _{2},\dots ,\psi _{\nu }$ by the formula:
\[
H_{k+j}\left( \psi _{r}^{\left( j,m\right) }\left( z\right) ;t\left( \nu
+1\right) \right) =m_{r}\left( t\right) H_{0}\left( \frac{1}{z};t\right) .
\]

Then we have that the system%
\[
\left\{ \psi _{r}^{\left( j,m\right) }\left( z\right) \right\} _{j,m}
\]%
is not an orthogonal set with respect to Lebesgue measure on $\textbf{R}$, so
$$
\left\{ \psi _{r}^{\left( j,m\right) }\left(
z\right) \right\} _{j,m}
$$
is not an orthogonal basis for $L^{2}\left( {\bf R}\right) $, but only a tight frame in the sense that
\[
\sum_{m,j,r}\left| \left\langle f\biggm|\psi _{r}^{\left( j,m\right) }\left(
z\right) \right\rangle \right| =\left\| f\right\| ^{2}
\]%
for all $f\in L^{2}\left( {\bf R}\right) $.

Let us specialize to the following case on the space $L^{2}\left( {\bf T,}\mu \right)$ with
$d\mu \left( z\right) =z^{-1}\,dz $:
\begin{eqnarray*}
m_{0}(z) &=&\sum_{k\in {\bf Z}}b_{k}J_{k}(z)\text{\quad and\quad }%
m_{r}\left( \sigma _{j}\left( z\right) \right) =\sum_{k\in {\bf Z}%
}b_{k}J_{k+r}\left( ze^{\pi ij}\right)  \\
&&\qquad \text{where }\sigma _{j}\left( z\right) =\sigma _{0}\left( z\right)
e^{\pi ij}.
\end{eqnarray*}

The unitarity of the matrix $M\left( z\right) $ implies the following
conditions:

{\bf 1}. For the diagonal entries we have:
\[
\sum_{r=0}^{\nu }\left| m_{r}\left( ze^{\pi ir}\right) \right|
^{2}=\sum_{r=0}^{\nu }\sum_{k,l}b_{k}\overline{b_{l}}
J_{k+r}(z)\overline{J_{l+r}(z)}e^{\pi ir\left( k-l\right) }.
\]

Since we have the following:
\[
\int_{\left| z\right| =1}\sum_{r=0}^{\nu }\left| m_{0}\left( \sigma
_{r}\left( z\right) \right) \right| ^{2}\,d\mu \left( z\right) =1=\frac{1}{%
2\pi i}\int_{\left| z\right| =1}d\mu \left( z\right) ,
\]%
for $k=l$ the
Residue Theorem gives the following:
\[
1=
\left( \nu +1\right)
\sum_{k}\left| b_{k}\right| ^{2}\frac{1}{\left( k!2^{k}\right) ^{2}}.
\]

{\bf 2.} For the off-diagonal entries, i.e., for
$k^{\prime }\neq l^{\prime }$, we have:%
\[
\sum_{r=0}^{\nu }m_{k^{\prime }}\left( te^{\pi ir}\right)
\overline{m_{l^{\prime }}\left(
te^{\pi ir}\right) }=0\,;
\]%
then we get $\sum_{k^{\prime },l^{\prime }}b_{k^{\prime }}\overline{c_{l^{\prime }}}
=0$ and then $\sum_{k^{\prime }}b_{k^{\prime }}\overline{c_{n+k^{\prime }}}=0$,
$l^{\prime }=n+k^{\prime }$, by using the ``multiplicative periodicity'' of the
Bessel functions with respect to the argument. Define now $\psi _{1},\psi _{2},
\dots ,\psi _{\nu }$ such that%
\begin{equation}
\left( \nu +1\right) ^{\frac{1}{2}}H_{k+j}\left( \psi _{r}\left( \left(
z\right) ;\left( \nu +1\right) t\right) \right) =m_{j}\left( t\right)
H_{0}\left( \frac{1}{z};t\right) .  \label{nuova-3}
\end{equation}%
( such $\psi$ exist by the above wavelet construction).
Then the $\left\{ \psi _{r}^{\left( j,m\right) }\left( z\right) \right\}
_{j,m}$ are not orthogonal in $L^{2}\left( {\bf R}%
\right) $ but they satisfy:
\[
\sum_{m,j,r}\left| \left\langle f\biggm|\psi _{r}^{\left( j,m\right) }\left(
z\right) \right\rangle \right| =\left\| f\right\| ^{2}
\]%
for all $f\in L^{2}\left( {\bf R}\right) $. This follows as in Ref.\
\CITE{dau},
Prop.~6.2.3, from the unitarity of the matrix of the $\left( m_{i,j}\right)
_{i,j}=\left( m_{i}\left( \sigma _{j}\left( z\right) \right) \right) _{i,j}$
and from the formula (\ref{nuova-3}). It then follows that the set $\left\{
\psi _{r}^{\left( j,m\right) }\left( z\right) \right\} _{j,m}$ is a tight
frame.

Let us look at the case of the deformed representations of the algebra $%
O_{\nu +1}$.
See Ref.\ \CITE{Pa-07} for a class of deformed
representations of the Cuntz algebra related to
the Jackson $q$-Bessel functions.
We consider the space $L^{2}\left( {\bf T,}\nu \right) $ as
before, where we take the measure given by
\[
d\nu \left( z\right) =z^{-1}\,dz,
\]
and using the $q$-Bessel functions previously defined instead of
classical Bessel functions.

Define the operators $S_{k}$ on $L^{2}\left( {\bf T},d\mu \right) $ by:
\[
\left( S_{k}\xi \right) \left( z\right) =m_{k}\left( z\right) \xi \left(
z^{\nu +1}\right) ,
\]%
where
\[
m_{0}(z)=\sum_{k\in {\bf Z}}b_{k}J_{k}(z;q)
\]%
and
\[
m_{r}\left( \sigma _{j}\left( z\right) \right) =\sum_{k\in {\bf Z}%
}b_{k}J_{k+r}\left( zq^{j};q\right) .
\]

Hence we have:%
\[
\left( S_{k}^{\ast }\xi \right) \left( z\right) =\sum_{r\in {\bf Z}_{\nu
+1}}\rho _{r}\overline{m_{k}\left( \sigma _{r}\left( z\right) \right) }\xi
\left( \sigma _{r}\left( z\right) \right) .
\]

Thus:
\texttt{\begin{eqnarray*}
\left( S_{k}^{\ast }S_{k^{\prime }}\xi \right) \left( z\right) &=&\sum_{r\in
{\bf Z}_{\nu +1}}\rho _{r}\overline{m_{k}
\left( \sigma _{r}\left( z\right) \right) }m_{k^{\prime }}
\left( \sigma _{r}\left( z\right) \right)
\xi \left( \sigma \sigma _{r}\left( z\right) \right)
\\
\ &=&\delta _{k,k^{\prime }}
\xi \left(
z\right)
\end{eqnarray*}}%
by using the unitarity of the matrix $M\left( z\right) $.

\section{\label{MCR}Markov chains and  representations
of $O_{N}$ and  $SO_q{(N)}$}

Let $(\Omega,\textrm{\slshape }F,P)$ be a given probability space and let $S=\textbf Z_N$ be the finite set ${0,1,...N}$.
An $S$-valued sequence of random variables $\xi_n$, $n\in\textbf N$ is called an $S$-valued Markov chain if for every $n\in\textbf N$ and all $s\in S$ we have:
\begin{eqnarray}\label{eqn-1}
P(\xi_{n+1}=s\mid \xi_0,\ldots \xi_n)=P(\xi_{n+1}=s\mid \xi_n). \label{eqCunpound.0}
\end{eqnarray}
where $P(\xi_{n+1}=s \mid \xi_0,\ldots \xi_n)$ denotes the conditional probability of the event $(\xi_n=s)$
with respect to the random variable $\xi_n$ and respectively to the field generated by the $\xi_n$ which we denote by
	$\sigma(\xi_n)$.
	Similarly, $P(\xi_{n+1}=s\mid \xi_0,\ldots \xi_n)$ is the conditional probability of $\xi_{n+1}=s$ with respect to $\sigma(\xi_0,\ldots \xi_n)$, the $\sigma$-field generated by $\xi_0,\ldots \xi_n$
	Formula (\ref{eqCunpound.0})is the Markov property of the chain $\xi_n$, $n\in \textbf N$. The set $S$ is called  the state space and the elements of $S$ are called the states.
	We construct a model associated to representations of the Cuntz algebra $O_N$ which is a Markov chain. The transition probabilities depend on a parameter $0<q<1$. The Markov chain $\textbf{P}$ gives rise to a random walk on the quantum group $SO_q{(N)}$.
	Let us start by constructing the Markov chain we are interested in. Denote by $M:={\left\{\xi_n \right\}}$ the following process where the  $\xi_n$ are random variables with state space $S=\textbf Z_N={{0,\ldots, N}}$. We define the following transition probabilities:
	\begin{eqnarray}
	p(r\mid s)=P({{\xi_1=r\mid \xi_0=s}}).  
	\end{eqnarray}
	as in the following transition matrix:
	\begin{displaymath}
	\mathbf{P}=
	{(\left[N\right]_q)}^{-1}\left(\begin{array}{ccccc}
	q&q^{N}&q^{N-1}&\ldots &q^{2}\\
	q^{2}&q&q^{N}&\ldots &q^{3}\\
	\vdots&\vdots&\vdots&\vdots\\
	q^{N}&q^{N-1}&q^{N-2}&\ldots &q\\
\end{array}\right)
\end{displaymath}
The matrix $P$ is  doubly stochastic since
  \[
  \sum_{s \in S} p(r\mid s)=1 
  \]
  and
  \[ 
  \sum_{r \in  S} p(r\mid s)=1
 \]

The Markov property  is clearly satisfied by construction. The transition probabilities can be written as
\[
p(r\mid s)=q^{\sigma_{s}{\left(r\right)}}
   \]

  where $ \sigma_{s}{\left(r\right)}=N+r-s+1$  mod N

\section{\label{ISP}Iterated subdivisions and projection valued measures}

Let us consider the family of representations of the Cuntz algebra $O_N$ where $N=\nu+1$ previously constructed. A given representation of $O_N$ restricted to its canonical maximal abelian subalgebra $C{(X)}$ for $X$ a Gelfand space induces naturally a projection-valued measure on $X$. The isometries generating $O_N$ provide subdivisions of the Hilbert space ${H}$ in view of
\begin{eqnarray*}
S_{i}^{\ast }S_{j}=\delta _{ij}{\bf 1} 
\end{eqnarray*}
and
\begin{eqnarray*}
\sum_{i=0}^{N }S_{i}S_{i}^{\ast }={\bf 1}. 
\end{eqnarray*}
In particular for every $k\in \textbf {N}$ the subspaces :
\begin{equation}
H{(a_1,a_2,\ldots,a_k)}:=S_{a_1}S_{a_2}\ldots S_{a_k}H
\end{equation}
are mutually orthogonal and
\begin{equation}
\sum_{i_1,i_2,\ldots,i_k}H{(a_1,a_2,\ldots,a_k)}:=H
\end{equation}
If $f\in H$ and $\left\|f\right\|=1$ then
\begin{equation}
\mu_{f}{(.)}:=<f,E(.)f>=\left\|E{(.)f}\right\|
\end{equation}
is a probability measure on the unit interval $\left[0,1\right]$.
We want to specialize $E{(.)}$ to our case and compute this measure which turns out to be related to the Markov chain constructed before.
Let us observe that the index labels ${(a_1,a_2,\ldots,a_k)}$ are used to assign $N$-adic
partitions (e.g. the intervals $\left[\frac{a_1}{N}+\ldots+\frac{a_k}{N^k},\frac{a_1}{N}+
+\ldots+\frac{a_k}{N^k}+\frac{1}{N^k}\right]$), then we have the mapping
\begin{equation}
{(a_1,\ldots, a_k)}\rightarrow H{(a_1,a_2,\ldots,a_k)}
\end{equation}
where the $(a_1,a_2,\ldots,a_k)\in \{(0,1,\ldots, N)\}$ and the length of the interval is $\frac {1}{N^k}$.
These partitions are a special case of endomorphisms
\begin{equation}
\sigma : X \rightarrow X
\end{equation}
 where $X$ is a compact Hausdorff space and $\sigma$ is  continuous and onto. Then for every $x\in X$ we have that $card{(\sigma^{-1}(x))}=\left\{x\in X/ \sigma{(y)}=x\right\}=N$.
There exists branches of the inverse, i.e. maps
\begin{equation}
\sigma_0,\ldots,\sigma_{N-1} : X \rightarrow X\\
\end{equation}
 such that
 \begin{equation}
   \sigma\circ\sigma_i=1_X
 \end{equation}
 for  each $0\leq  i<N$
the above intervals written in terms of the maps are:
\begin{eqnarray*}
I_k{(a)}=\left[\frac{a_1}{N}+\frac{a_2}{N^2}+\ldots+\frac{a_k}{N^k},\frac{a_1}{N}+
\frac{a_2}{N^2}+\ldots+\frac{a_k}{N^k}+\frac{1}{N^k}\right]\\
=\sigma_{a_1}\circ\sigma_{a_2}\ldots\circ\sigma_{a_k}\left(X\right)
\end{eqnarray*}
The system $\sigma_a=\sigma_{a_1}\circ\sigma_{a_2}\ldots\circ\sigma_{a_k}$
forms a set of branches for $\sigma^{k}=\sigma \circ\sigma \ldots\circ\sigma $.
and is called an $N$- adic systems of partitions of $X$. Thus for every $k\in \textbf Z_+$ 
$\left\{ J_k\left( a\right) \right\}$ 
is a partition indexed by $a\in\Gamma_N^{k}:\Gamma_N\times\Gamma_N\times\ldots\times\Gamma_N$.
On the other hand, given an Hilbert space $H$, a partition of projections in $H$ is a system
${P{(i)}}_{i\in I}$ of projections, i.e. $P{(i)}=P{(i)}^{\ast}=P{(i)}^2$ such that
\[
P{(i)}P{(j)}=0
\]
 if $i\neq j$ and
\[
\sum_{i\in \textbf I}P{(i)}=1_H
\]
Let $N\in \textbf{N}$, $N \geq 2$. Suppose that for every $k\in \textbf{N}$, there is a partition of projections ${P_k{(a)}}_{a\in \Gamma_N^k}$ such that every ${P_{k+1}{(a)}}$ is contained in some ${P_{k+1}{(b)}}$ i.e.
${P_{k}{(b)}}{P_{k+1}{(a)}}={P_{k+1}{(a)}}$ then ${P_k{(a)}}_{a\in \Gamma_N^k}$ is  a system of partitions of
$1_H$.
By Lemma 3.5  \CITE{jo} given an $N$-adic system of projections of $X$
and ${P_k{(a)}}_{{k\in \textbf Z_+},{a\in \Gamma_N^k}}$ an $N$ adic system of projections there is a unique normalized orthogonal projection-valued measure $E{(.)}$ defined on the Borel subsets of $X$ with values in the orthogonal projections of $H$ such that $E{(J_k{a})}={P_k{(a)}}$ for every $k\in \textbf Z_+$, $a\in\Gamma_N^k$.
Let $S_{i}$ be a representation of $O_N$ on $H$ and let $a={(a_1,a_2,\ldots,a_k)}\in \Gamma_N^k$
and $S_a:=S_{(a_1)}\ldots S_{(a_k)}$ then ${P_k{(a)}=S_{(a)}S_{(a)}^{*}}$.
Assuming then unitarity condition on the filters ${m_j}$ we get
\begin{eqnarray*}
\mu_{f}\left( I_k(a)\right)=\left| E{(I_k(a))}\right|^2=\left| S_{(a)}S_{(a)^*}f\right|^2\\
=<f,S_{(a)}S_{(a)}^{\ast}f>=\left\|S_{(a)}^* f \right\|^2\\
= \sum\left|<e_n,S_{(a)}^{\ast}f>\right|^2=\sum\left|<S_{(a)}e_n,f>\right|^2
\end{eqnarray*}
Using Plancherel theorem for Hankel transforms for $a=a_1$ we get that
\begin{equation}
\mu_{f}\left(I_1(a)\right)=
\sum_{j\in \textbf Z_N}\left|<H^q{\psi _a(z),q^j tN},f>\right|^2=\left(\left[N\right]_q^2\right)^{-1}
\end{equation}
Choosing $2j=N-r-s+1$ $mod$ $N$  we get that $<S_{a_{\left( j \right) }}e_n,f>$  gives
transition probabilities of the Markov chain constructed in the previous section.
\section{\label{MT}Markov trace and representations of the braid group $B_{\infty}$}
Let $\mathcal{F}$ be a category whose objects we denote by $\rho$, $\sigma$,
$\tau,\;\ldots$. The set of arrows between a
pair $\rho$, $\sigma $ of objects will be denoted by $(\rho ,\sigma)$ and
the identity of $\rho $ by $1_\rho $.  A BWM symmetry is a linear operator $G$ on $\rho \otimes \rho $
satisfying the Yang--Baxter equation
\[
G_1G_2G_1=G_2G_1G_2,
\]
and the following BWM condition: let $E=1-( q-q^{-1}) ^{-1}( G-G^{-1}) $.
Then
\[
EG=p^{-1}E,\qquad EGE=pE,\qquad EG^{-1}E=p^{-1}E ,
\]
where $p, q\in \mathbf{C-}\{ 0\} $ (to be specified later) and $G $ satisfy
the cubic equation
\[
( G-q) ( G+q^{-1}) ( G-p^{-1}) =0.
\]
Then $E$ is a complex multiple of a projection:
\[
E^2=( 1+( p-p^{-1}) ( q-q^{-1}) ^{-1}) E.
\]
In particular for our purpose let us consider $( \mathcal{F},G) $   a braided tensor C*-category associate to the quantum group $SO(N)$ \CITE{Pa-La}.
Let $g\in B_{\infty}$ be an element of the infinite braid group and let $p=p(g)$ be its associated
permutation written as a product of disjoint cycles of length $k_1,\ldots,k_m$ with $k_1+k_2+\ldots+k_m=n$. Denote by $\theta$ the braiding in the category.
Then
\begin{equation}
\omega^{(n)}\left({\theta}^{(n)}(g)\right)=\phi^{(n)}{\theta}^{(n)}(g)=
{d_q(\rho)}^{(m)}\left(\phi\left(\theta\left(\rho,\rho\right)\right)\otimes1_{{\rho}^{n-1}}\right)^{n}
\end{equation}
where 
$\phi ( T) =C^{*}\otimes 1_{\rho ^{n-1}}\circ 1_{\bar \rho }\otimes T\circ
C\otimes 1_{\rho ^{n-1}}$,
${(d_q(\rho)}=C^{\ast}\circ C$,  $C\in {(i,\rho\overline{\rho})}$ and
$\overline{C}\in {(i,\overline{\rho}\rho)}$ are intertwiners. For our purpose we let  $( \mathcal{F},G) $  be a braided tensor C*-category associate to the quantum group $SO(N)$ \CITE{Pa-La}
generated by a single object $\rho =H$ and having conjugate $\bar \rho $.
By $ \omega^{(n)}$ we denote the Markov trace for the BWM symmetries which has modulus
 $q^{(2m)}(d_q(\rho)^{(-1)}$ where $(d_q(\rho)$ is the quantum dimension.

For the quantum $\mathrm{SO}( N) $ (see \cite{Pa-La}), $N=2m+1$, the
operator $G$ has the form
\begin{eqnarray}
G &=&\sum_{i\ne 0}( qe_{i,i}\otimes e_{i,i}+q^{-1}e_{i,-i}\otimes e_{-i,i})
+e_{0,0}\otimes e_{0,0}+\sum_{i\ne j}e_{i,j}\otimes e_{j,i}  \nonumber \\
&&\qquad +( q-q^{-1}) ( \sum_{i<j}e_{i,i}\otimes e_{j,j}-\sum_{j<i}q^{\frac{%
i+j}2}e_{i,j}\otimes e_{-i,-j}) .  \label{G-def}
\end{eqnarray}
Here $\{ e_{i,j}\} $ is the $N\times N$ matrix with $1$ in the $( i,j) $
position and $0$ elsewhere; $G$ acts on a finite-dimensional Hilbert space $H
$ with basis indexed by $I=\{ -2m+1,-2m+3,...-3,-1,0,1,3,...,2m-1\} $. The
element $E=1-( q-q^{-1}) ^{-1}( G-G^{-1}) $ has the form
\[
E=\sum_{i,j}q^{\frac{i+j}2}e_{i,j}\otimes e_{-i,-j}.
\]
Then it is easy to see that $E^2=xE$, $x=\sum_iq^i$.

By \CITE{Pa-La}

\begin{itemize}
\item[(\emph{i})]  \textsl{There exists a faithful Markov trace $w$ given by
a left inverse via a conjugate $C\in ( i,\bar \rho \rho ) $ such that
$
w( G) =\frac{q^{2m}}{d_q( \rho ) }
$
and $E=C\circ C^{*}$ such that $E=( q-q^{-1}) ^{-1}( G-G^{-1}) $.}
\item[(\emph{ii})]  \textsl{There exists $\tau _q\in ( \mathbf{C},\rho ^2)$,
a group-like element, and non-degenerate mapping given by
$\tau _q\lambda=\lambda \sum_ie_i\otimes J^{-1}e_i$,
where $J=( q^{j/2}\delta _{i,\bar j}) $,
$\bar j=N+1-j$. Furthermore there exists an antisymmetric tensor
$\epsilon_{i_1\ldots i_N}:\mathbf{C}\longrightarrow H^N$
which gives a non-degenerate form.}
\end{itemize}

Thus we construct a random walk on $SO_q{(N)}$ induced from the Markov chain as follows:
choose $2m=N+j-i+1$ $mod$ $N$ and $d_q{{(\rho)}}=\left[N\right]_{(q^2)}$ presented in section \ref{MCR}.
Thus
\begin{equation}
w( G) =\frac{q^{2m}}{d_q( \rho ) }=p(j,i)
\end{equation}
Thus the transition probabilities $p(j,i)$ of the Markov chain give rise to a Markov trace on $SO_q(N)$ with
$N=2m+1$

{\textbf{Acknowledgments}}
P.J. was supported in part by a grant from the National Science Foundation (USA).
A.M.P. would like to thank the Max-Planck Institut f\"ur Mathematik 
in Bonn for support and excellent working conditions.


\end{document}